\documentclass[11pt,a4paper]{article}
\usepackage[full]{textcomp}
\usepackage{newtxtext} 
\usepackage[utf8]{inputenc}

\usepackage{graphicx,amsmath,amssymb} 
\usepackage{mathtools}
\usepackage{enumitem}
\setlist{topsep=1pt,parsep=1pt,itemsep=1pt} 
\usepackage{tikz}
\usetikzlibrary{matrix,arrows}
\usetikzlibrary{decorations.pathmorphing}
\usetikzlibrary{decorations.markings}
\usepgflibrary{fpu}
\usepackage{xcolor}
\usepackage{mathrsfs} 
\usepackage{bbm}

\usepackage{cleveref}
\usepackage{amsthm}  
\usepackage{autonum} 

\usepackage{thmtools}

\input{sr_spde.sty}
\numberwithin{equation}{section}

\title{Concentration estimates for slowly time-dependent\\ 
singular SPDEs on the two-dimensional torus}
\author{Nils Berglund and Rita Nader}
\date{}

\begin{document}

\maketitle

\begin{abstract}
We consider slowly time-dependent singular stochastic partial differential 
equations on the two-dimensional torus, driven by weak space-time white noise, 
and renormalised in the Wick sense. Our main results are concentration results 
on sample paths near stable equilibrium branches of the equation without noise, 
measured in appropriate Besov and H\"older norms. We also discuss a case 
involving a pitchfork bifurcation. These results extend to the two-dimensional 
torus those obtained in~\cite{BG_pitchfork} for finite-dimensional SDEs, and 
in~\cite{Berglund_Nader_22} for SPDEs on the one-dimensional torus.
\end{abstract}

\leftline{\small{\it Date.\/} September 30, 2022. Updated October 22, 2022. 
}
\leftline{\small 2020 {\it Mathematical Subject Classification.\/} 
60H15, 		
60G17		
(primary),
34F15,   	
37H20   	
(secondary)}


\section{Introduction}
\label{sec:intro} 

In this work, we are interested in slowly time-dependent singular 
stochastic partial differential equations (SPDEs) on the two-dimensional torus, 
of the form 
\begin{equation}
\label{eq:SPDE_intro} 
 \6\phi(t,x) = \bigbrak{\Delta\phi(t,x) + \Wick{F(\eps t, \phi(t,x))}} \6t 
 + \sigma \6W(t,x)\;,
\end{equation} 
where $\Wick{F}$ denotes Wick renormalisation (see below), and $\6W(t,x)$ 
denotes space-time white noise. While analogous SPDEs on the one-dimensional 
torus are well-posed, without the need for any renormalisation procedure, it is 
well known that renormalisation is required in dimension two and higher, 
because space-time white noise is a distribution-valued process that is too 
singular. 

The well-posedness problem on the two-dimensional torus was first solved 
by Giuseppe Da Prato and Arnaud Debussche in the landmark 
work~\cite{daPratoDebussche}. The main idea of their approach is to write an 
equation for the difference between the solution and the stochastic 
convolution, which solves a linear equation. It turns out that unlike the 
stochastic convolution, which is distribution-valued, the difference is an 
actual function. Solutions to the equation can then be constructed by a 
fixed-point argument in an appropriate Besov space, provided the equation is 
renormalised in the sense of Wick. While the method has been spelled out for 
time-independent systems, extending it to time-dependent equations of the 
form~\eqref{eq:SPDE_intro} is straightforward. 

The work~\cite{daPratoDebussche} has later given rise to far-reaching 
generalisations, that allow to solve large classes of singular SPDEs. These 
generalisations include the theory of regularity structures, introduced by 
Martin Hairer in the work~\cite{Hairer2014} and further developed with Ajay 
Chandra, Yvain Bruned, Ilya Chevyrev and Lorenzo Zambotti 
in~\cite{BrunedHairerZambotti,ChandraHairer16,
Bruned_Chandra_Chevyrev_Hairer_18}, and the theory of paracontrolled 
distributions, introduced in~\cite{Gubinelli_Imkeller_Perkowski_15} by 
Massimiliano Gubinelli, Peter Imkeller and Nicholas Perkowski. 
Most of these more general singular SPDEs require more refined renormalisation 
methods than Wick renormalisation. 

For time-independent versions of the equation~\eqref{eq:SPDE_intro} on the 
two-dimensional torus, many results going beyond well-posedness and 
existence/uniqueness of solutions have been obtained. For instance, the fact 
that their solutions satisfy the Markov property and are reversible with 
respect to the Gibbs measure was proved in~\cite{Rockner_Zhu_Zhu_15} using 
Dirichlet forms, while uniqueness of the Gibbs measure and convergence to it 
were obtained in~\cite{Rockner_Zhu_Zhu_17}. The fact that solutions satisfy the 
strong Feller property and are exponentially mixing was shown 
in~\cite{Tsatsoulis_Weber_16} using a dissipative bound, while the strong 
Feller property was also proved (for more general equations) 
in~\cite{Hairer_Mattingly_18}, using the theory of regularity structures.  
The work~\cite{HairerWeber} provided a large-deviation principle, valid for a 
class of two- and three-dimensional singular SPDEs. In the particular case of 
the Allen--Cahn equation, sharper asymptotics on transition times between 
metastable states than those provided by large-deviation estimates have been 
obtained in~\cite{Berglund_DiGesu_Weber_16} and~\cite{Tsatsoulis_Weber_18}. 

In the present work, we are interested in obtaining more detailed 
non-equilibrium properties for time-dependent SPDEs of the 
form~\eqref{eq:SPDE_intro} on the two-dimensional torus. The case of the 
one-dimensional torus has been previously considered in the 
work~\cite{Berglund_Nader_22}. The first main result of that work concerned the 
motion near so-called stable equilibrium branches of the equation. These are 
curves of the form $t\mapsto\phi^*(t,x)$ on which the right-hand side of the 
equation vanishes in the absence of noise. The deterministic equation admits 
particular solutions that stay at distance of order $\eps$, in the $H^1$ Sobolev 
norm, from $\phi^*$, and it was proved that solutions of the stochastic equation 
remain with high probability in a neighborhood of size of order $\sigma$, 
measured in the $H^s$ Sobolev norm for $s < \frac12$. This result provides an 
extension to the infinite-dimensional setting of similar results previously 
obtained in~\cite{berglund2002GeoPerSDE,Berglund_Gentz_book} for 
finite-dimensional stochastic differential equations. 

The other results in~\cite{Berglund_Nader_22} concerned certain situations 
involving bifurcations, or avoided bifurcations. These occur when the 
equilibrium branch $t\mapsto\phi^*(t,x)$ (almost) loses stability at some time, 
usually because of the presence of a nearby unstable equilibrium branch. This 
can result in interesting phenomena such as stochastic resonance, where 
solutions of the equation make fast jumps in a close-to-periodic way. Those 
results were an infinite-dimensional generalisation of one-dimensional results 
obtained in~\cite{BG_SR}.

The aim of the present work is to obtain similar results in the case of the 
two-dimensional torus, where Wick renormalisation is needed. The main result, 
Theorem~\ref{thm:stoch_convolution}, shows that Wick powers of the stochastic 
convolution remain concentrated near zero with high probability. Their size is 
measured here in the Besov norm $\smash{\normBesov{\cdot}{\alpha}{2}{\infty}}$ 
for any parameter $\alpha < 0$. Theorem~\ref{thm:phi1} shows that this estimate 
implies concentration properties of solutions in a neighborhood of a 
stable equilibrium branch $\set{\phi^*(t,x)}_{0\leqs t\leqs T}$. In particular, 
the difference between a solution of~\eqref{eq:SPDE_intro} and the stochastic 
convolution is likely to remain small, in a stronger H\"older norm of positive 
index. 

Despite this concentration result, one may be concerned that it is of little 
practical use, because it does only concern the difference between a solution 
and the more singular stochastic convolution. Theorems~\ref{thm:phi1perp} 
and~\ref{thm:phi10} show that this is not the case, by discussing the 
particular situation of a dynamic pitchfork bifurcation, which was previously 
considered in~\cite{BG_pitchfork} for one-dimensional stochastic differential 
equations. 

The remainder of this paper is organized as follows. Section~\ref{sec:results} 
contains a precise description of the set-up, a short introduction to Besov 
spaces, and the concentration results in the stable case, and in a case 
involving a pitchfork bifurcation. Section~\ref{sec:proof_stoch_convolution} 
contains the proof of Theorem~\ref{thm:stoch_convolution} on Wick powers of the 
stochastic convolution. Section~\ref{sec:proofs_concentration} contains the 
proofs of the other concentration results. Three appendices provide further 
information on Besov spaces, Wick calculus, and some technical proofs.

\subsection*{Acknowledgments}

The authors would like to thank Tom Klose for pointing out 
reference~\cite{Janson_book}. 
This work is supported by the ANR project PERISTOCH, ANR–19–CE40–0023.


\section{Set-up and main results}
\label{sec:results} 


\subsection{A family of Wick-renormalised singular SPDEs}
\label{ssec:Wick} 

We are interested in renormalised versions of the SPDE 
\begin{equation}
\label{eq:SPDE0} 
 \6\phi(t,x) = \bigbrak{\Delta\phi(t,x) + F(\eps t, \phi(t,x))} \6t 
 + \sigma \6W(t,x)\;,
\end{equation} 
where time $t$ belongs to an interval $I=[0,T]\subset\R_+$, the spatial variable 
$x$ belongs to the two-dimensional torus $\T^2 = (\R/\Z)^2$, and the solution 
$\phi(t,x)$ is real-valued. In addition, we assume that 
\begin{itemize}
\item 	$\eps > 0$ and $\sigma\geqs0$ are small positive parameters;
\item 	$F$ is polynomial, of the form 
\begin{equation}
\label{eq:def_F} 
 F(t,\phi) = \sum_{j=0}^n A_j(t) \phi^j 
\end{equation} 
for some odd $n\geqs3$, where the coefficients $A_j:I\to\R$ are of class 
$\cC^1$, and the leading coefficient $A_n(t)$ is strictly negative for all 
$t\in I$, to avoid blow-up of solutions;
\item 	$\6W(t,x)$ denotes space-time white noise on $I\times\T^2$. 
\end{itemize}

It is well-known (see for instance~\cite{daPratoDebussche}) that the 
SPDE~\eqref{eq:SPDE0} is not well-posed, and that a renormalisation procedure 
is required to define a notion of solution. There exist several slightly 
different ways of doing this. For our purposes, it will be convenient to work 
with spectral Galerkin approximations. Let 
\begin{equation}
 \bigset{e_k(x) = \e^{2\pi\icx k\cdot x}}_{k \in \Z^2}
\end{equation} 
denote a complex Fourier basis of $L^2(\T^2)$, and write any $\phi\in 
L^2(\T^2)$ as 
\begin{equation}
\label{eq:Fourier} 
 \phi(x) = \sum_{k\in\Z^2} \phi_k e_k(x)\;.
\end{equation} 
Note that since $\phi(x)$ is assumed to be real, the coefficients $\phi_k$ 
satisfy the reality condition 
\begin{equation}
 \label{eq:reality_cond}
 \phi_{-k} = \overline{\phi_k} \qquad \forall k\in\Z\;.
\end{equation} 
For any cut-off $N\in\N$, we define the spectral Galerkin approximation at 
order $N$ of $\phi$ by 
\begin{equation}
 \phi_N(x) = (P_N\phi)(x)
 := \sum_{k\in\Z^2 \colon \abs{k}\leqs N} \phi_k e_k(x)\;,
\end{equation} 
where $\abs{k} = \abs{k_1} + \abs{k_2}$. We denote the eigenvalues of the 
Laplacian on $\T^2$ by $-\mu_k$, with
\begin{equation}
\label{eq:ev_Laplacian} 
 \mu_k := (2\pi)^2 \norm{k}^2\;, \qquad k\in\Z^2\;,
\end{equation} 
where $\norm{k}$ denotes the Euclidean norm of $k$, and define the 
renormalisation constant 
\begin{equation}
\label{eq:def_CN} 
 C_N = \frac{\sigma^2}{2}\Tr\Bigpar{\brak{-P_N\Delta+1}^{-1}}
 = \sigma^2\sum_{k\in\Z^2 \colon \abs{k}\leqs N} \frac{1}{2(\mu_k + 1)}\;.
\end{equation} 
One easily checks that $C_N$ diverges like $\sigma^2\log N/(2\pi)$ as 
$N\to\infty$. Note that the shift $+1$ in the definition~\eqref{eq:def_CN} of 
$C_N$ is only there to avoid problems with the $k=0$ mode, and can be replaced 
by any other strictly positive constant. 

Recall that the Hermite polynomials with variance $C_N$ are defined recursively 
by 
\begin{equation}
 H_0(x; C_N) := 1\;, \qquad 
 H_{m+1}(x; C_N) := x H_m(x; C_N) - C_N \frac{\6}{\6x}H_m(x; C_N)
 \quad \forall m\in\N_0\;.
\end{equation} 
The $m$th Wick power of $\phi_N$ is defined by
\begin{equation}
 \Wick{\phi_N^m} 
 = \Wick{\phi_N^m}_{C_N}
 := H_m(\phi_N; C_N)\;.
\end{equation} 
For instance, we have 
\begin{align}
 \Wick{\phi_N(x)^1} &= \phi_N(x)\;, \\
 \Wick{\phi_N(x)^2} &= \phi_N(x)^2 - C_N\;, \\
 \Wick{\phi_N(x)^3} &= \phi_N(x)^3 - 3C_N\phi_N(x), \\
 \Wick{\phi_N(x)^4} &= \phi_N(x)^4 - 6C_N\phi_N(x)^2 + 3C_N^2\;.
\end{align} 
The renormalised version of~\eqref{eq:SPDE0} we want to study is given by the 
limit, as $N\to\infty$, of 
\begin{equation}
\label{eq:SPDE_renorm_N} 
 \6\phi_N(t,x) = \bigbrak{\Delta\phi_N(t,x) + \Wick{F(\eps t, 
\phi_N(t,x))}_{C_N}} \6t + \sigma \6W_N(t,x)\;,
\end{equation} 
where $\6W_N = P_N\6W$, and 
\begin{equation}
 \Wick{F(t,\phi)}_{C_N} 
 := \sum_{j=0}^n A_j(t) \Wick{\phi^j}_{C_N}\;. 
\end{equation} 
As proved in~\cite{daPratoDebussche}, solutions of the renormalised 
equation~\eqref{eq:SPDE_renorm_N} do admit a well-defined limit as 
$N\to\infty$, in appropriate Besov spaces that we define below. The limiting 
equation is denoted by 
\begin{equation}
\label{eq:SPDE_renorm0} 
 \6\phi(t,x) = \bigbrak{\Delta\phi(t,x) + \Wick{F(\eps t, 
\phi(t,x))}} \6t + \sigma \6W(t,x)\;. 
\end{equation} 
In what follows, it will be convenient to rescale time by a factor $\eps$, 
which results in the SPDE 
\begin{equation}
\label{eq:SPDE_renorm} 
 \6\phi(t,x) = \frac{1}{\eps}\bigbrak{\Delta\phi(t,x) + \Wick{F(t, 
\phi(t,x))}} \6t + \frac{\sigma}{\sqrt{\eps}} \6W(t,x)\;. 
\end{equation} 


\subsection{Besov spaces and the \lq\lq Da Prato--Debussche trick\rq\rq}
\label{ssec:Besov} 

As mentioned above, the importance of Besov spaces in solving singular SPDEs of 
the form~\eqref{eq:SPDE_renorm} was  realised in the seminal 
work~\cite{daPratoDebussche}. We recall one of their definitions here.

\begin{definition}[Besov spaces]
\label{norm_besov}
Let $\phi$ admit the Fourier series~\eqref{eq:Fourier}.
We define a collection of annuli by setting $\cA_0 = \set{(0,0)}$
and $\cA_q = \setsuch{k\in\Z^2}{2^{q-1} \leqs \abs{k} < 2^q}$
for any $q\in\N$. The projection of $\phi$ on $\cA_q$ is defined by 
\begin{equation}
 \delta_q\phi(x) := \sum_{k\in\cA_q} \phi_k e_k(x)\;.
\end{equation} 
For $\alpha\in\R$ and $p, r\in[1,\infty]$, define the norm 
\begin{align}
 \normBesov{\phi}{\alpha}{p}{r} 
 &:= \bignorm{\bigset{2^{rq\alpha} 
\norm{\delta_q\phi}_{L^p}}_{q\geqs0}}_{\ell^r} \\
 &:= 
\begin{cases}
 \biggpar{\displaystyle
 \sum_{q\geqs0} 2^{rq\alpha} \norm{\delta_q\phi}_{L^p}}^{1/r}
 & \text{if $1\leqs r<\infty$\;,}\\[8pt]
 \displaystyle
 \sup_{q\geqs0} 2^{q\alpha} \norm{\delta_q\phi}_{L^p}
 & \text{if $r=\infty$\;.}
\end{cases}
\end{align} 
Then the Besov space $\Besovspace{\alpha}{p}{r} = 
\Besovspace{\alpha}{p}{r}(\T^2)$ is defined as the set of all $\phi$ such that 
$\normBesov{\phi}{\alpha}{p}{r} < \infty$. 
\end{definition}

The Besov space $\Besovspace{\alpha}{p}{r}$ is a Banach space for all 
$\alpha\in\R$ and $p, r\in[1,\infty]$. In particular, 
\begin{equation}
 \cC^\alpha := \Besovspace{\alpha}{\infty}{\infty}
 \qquad\text{and}\qquad 
 H^\alpha := \Besovspace{\alpha}{2}{2}
\end{equation} 
co\"incide with the usual H\"older and (fractional) Sobolev spaces 
respectively. 

We will use the following results, which can be found, for instance, 
in~\cite[Proposition~2.1]{daPratoDebussche}, in 
\cite[Lemma~3.3]{daPratoDebussche} and~\cite{Mourrat_Weber_17}.

\begin{prop}[Embeddings and products] \hfill
\label{prop:Besov} 
\begin{enumerate}
\item 	If $\alpha\in\R$, $1\leqs p_1\leqs p_2\leqs\infty$ and $1\leqs q_1\leqs 
q_2\leqs\infty$, then $\Besovspace{\alpha}{p_1}{q_1}$ is continuously embedded 
in $\Besovspace{\beta}{p_2}{q_2}$, where $\beta = \alpha - 
2\bigpar{\frac{1}{p_1} - \frac{1}{p_2}}$.

\item 	If $\alpha_1 < \alpha_2\in\R$ and $p, q\in[1,\infty]$, then 
$\Besovspace{\alpha_1}{p}{q}$ is compactly embedded 
in $\Besovspace{\alpha_2}{p}{q}$. 

\item 	Let $p, r \geqs1$ and let $\alpha, \beta < \frac2p$ satisfy 
$\alpha+\beta > 0$. Then, if $\phi\in\Besovspace{\alpha}{p}{r}$ and 
$\psi\in\Besovspace{\beta}{p}{r}$, one has
\begin{equation}
\label{eq:bound1_product_Besov} 
\phi \psi\in\Besovspace{\gamma}{p}{r}
\qquad\text{and}\qquad 
 \smallnormBesov{\phi\psi}{\gamma}{p}{r} 
 \lesssim \normBesov{\phi}{\alpha}{p}{r}
\normBesov{\psi}{\beta}{p}{r}\;,
\end{equation} 
where $\gamma = \alpha + \beta - \frac{2}{p}$.

\item 	Let $n\in\N$, $p, r\geqs1$ and $-\frac{2}{p(2n+1)} < \alpha < 0$. Set 
$s = \frac2p+2\alpha$. Then, if $\phi\in\Besovspace{s}{p}{r}$ and 
$\psi\in\Besovspace{\alpha}{p}{r}$, one has
\begin{equation}
\label{eq:bound2_product_Besov} 
\phi^\ell \psi\in\Besovspace{(2\ell+1)\alpha}{p}{r}
\qquad\text{and}\qquad 
 \smallnormBesov{\phi^\ell \psi}{(2\ell+1)\alpha}{p}{r} 
 \lesssim \normBesov{\phi}{s}{p}{r}^\ell 
\normBesov{\psi}{\alpha}{p}{r}
\end{equation} 
for $\ell\in\set{0,\dots,n-1}$, with a constant depending on $\alpha$, $s$, 
$p$, $r$ and $n$. 
\end{enumerate}
\end{prop}

Let $\psi$ denote the stochastic convolution, that is, the solution of the 
linear equation 
\begin{equation}
 \6\psi(t,x) = \Delta\psi(t,x) \6t + \sigma \6W(t,x)
\end{equation} 
with initial condition $\psi(0,x) = 0$. It is known (see, for instance, 
\cite[Lemma~3.2]{daPratoDebussche}) that the stochastic convolution belongs to 
all Besov spaces with \emph{negative} regularity $\alpha$, but not with 
positive $\alpha$. This means that $\psi$ is a distribution, but not a 
function. The central idea in~\cite{daPratoDebussche} is that the difference 
$\phi_1 = \phi - \psi$ enjoys much better regularity properties:

\begin{theorem}[{\protect\cite[Theorem~4.2]{daPratoDebussche}}]
\label{thm:daPratoDebussche} 
For any $p>n$ and $r\geqs1$, let $\alpha$ and $s$ satisfy 
\begin{equation}
 0 > \alpha > \max\biggset{-\frac{2}{p(n+1)}, 
-\frac{1}{n-1}\biggpar{1-\frac{n}{p}}}\;, \qquad 
 s = \frac{2}{p} + 2\alpha\;.
\end{equation} 
Then, for almost any initial condition (with respect to a natural probability 
measure), the renormalised SPDE admits for any $T\geqs0$ a 
unique solution $\phi$ such that 
\begin{equation}
 \phi - \psi \in \cC([0,T],\Besovspace{\alpha}{p}{r}) 
 \cap L^p([0,T], \Besovspace{s}{p}{r})\;.
\end{equation} 
\end{theorem}

Note in particular that $s > 0$, implying that the difference $\phi - \psi$ 
takes values in the space of functions $\Besovspace{s}{p}{r}$, which have some 
H\"older regularity in space.  


\subsection{Main results: Wick powers of the stochastic convolution}
\label{ssec:stoch_convol} 

Our first main results concern the stochastic convolution and its Wick powers. 
Let $a:I\to \R$ be a continuously differentiable function satisfying 
\begin{equation}
\label{eq:a(t)} 
 -a_+ < a(t) < -a_- \qquad \forall t\in I
\end{equation} 
for some constants $a_+ > a_- > 0$. The (time-inhomogeneous) stochastic 
convolution is defined as the solution of the linear equation
\begin{equation}
\label{eq:stoch_convolution} 
 \6\psi(t,x) = \frac{1}{\eps}\bigbrak{\Delta\psi(t,x) + 
 a(t) \psi(t,x)} \6t + \frac{\sigma}{\sqrt{\eps}} \6W(t,x)
\end{equation} 
with initial condition $\psi(0,x) = 0$ $\forall x\in\T^2$. The following 
estimate is the main result of this section. 

\begin{theorem}[Tail estimates on Wick powers of the stochastic convolution]
\label{thm:stoch_convolution} 
For any $\alpha < 0$ and for any $m\in\N$, there exist constants 
$C_m(T,\eps,\alpha)$ and $\kappa_m(\alpha)$, independent of the cut-off $N$, 
such that 
\begin{equation}
\label{eq:psi_tail_estimate}
\biggprob{\sup_{0\leqs t\leqs T} 
\normBesov{\Wick{\psi(t,\cdot)^m}}{\alpha}{2}{\infty} > h^m}
\leqs C_m(T,\eps,\alpha) \e^{-\kappa_m(\alpha) h^2/\sigma^2}
\end{equation} 
holds for all $h>0$. Furthermore, there are constants $c_0, c_1$, uniform in 
$m$, $\alpha$, $T$ and $\eps$, such that 
\begin{equation}
 \kappa_m(\alpha) \geqs c_0 \frac{\alpha^2}{m^7}\;, \qquad 
 C_m(T,\eps,\alpha) \leqs c_1 \frac{T}{\eps} 
\frac{m^{3/2}\e^mm^m}{\abs{\alpha}}\;.
\end{equation} 
\end{theorem}

\begin{remark}
Comparable results cannot be expected to hold in any Besov space 
$\Besovspace{\alpha}{p}{\infty}$. For instance, we have 
\begin{equation}
 \bigexpec{\normBesov{\psi(t,\cdot)}{\alpha}{\infty}{\infty}}
 = \sup_{q \geqs 0} 2^{q\alpha} \sum_{k\in\cA_q} 
 \bigexpec{\abs{\psi_k(t)}}\;.
\end{equation} 
Since the random variables $\psi_k(t)$ follow centred normal distributions of 
variance of order $\norm{k}^{-2}$, the sum over $k\in\cA_q$ of the expectations 
of $\abs{\psi_k(t)}$ has order $2^q$. Therefore, the expectation of 
$\normBesov{\psi(t,\cdot)}{\alpha}{\infty}{\infty}$ diverges with the cut-off 
$N$ as $N^{\alpha+1}$ if $\alpha > -1$. Since the limiting random variable does 
not admit a first moment, its tail probabilities have to decay more slowly than 
$1/h$. 
\end{remark}

\begin{remark}
The following observation may provide some intuition on what it means for a 
distribution to be concentrated in a ball in the Besov space 
$\Besovspace{\alpha}{2}{\infty}$. Let $\eta:\T^2\to\R$ be a compactly supported 
test function of class $\cC^1$, of unit $\cC^1$-norm, and set 
\begin{equation}
 \eta_\rho(x) = \frac{1}{\rho} \eta\biggpar{\frac{x}{\rho}}
\end{equation} 
for any $\rho\in(0,1]$. Note that the scaled test functions $\eta_\rho$ have 
constant $L^2$-norm, instead of constant $L^1$-norm, as one would require when 
working with $\Besovspace{\alpha}{\infty}{\infty}$. Then we have 
\begin{equation}
 \bigabs{\pscal{\Wick{\psi^m}}{\eta_{2^{-q_0}}}}
 \lesssim 2^{\abs{\alpha}q_0} \normBesov{\Wick{\psi^m}}{\alpha}{2}{\infty}\;,
\end{equation} 
for all $q_0\in\N_0$ (cf.~Lemma~\ref{lem:Besov_test}), so that 
Theorem~\ref{thm:stoch_convolution} implies 
\begin{equation}
 \biggprob{\sup_{0\leqs t\leqs T} 
\pscal{\Wick{\psi(t)^m}}{\eta_{2^{-q_0}}} > h^m}
\leqs C_m(T,\eps) \exp\biggset{-\kappa_m(\alpha) 
2^{-2\abs{\alpha}q_0/m}\frac{h^2}{\sigma^2}}
\end{equation} 
for any $m\in\N$ and any $q_0\in\N_0$. This shows that sample paths of 
$\pscal{\Wick{\psi(t)^m}}{\eta_{2^{-q_0}}}$ are concentrated in a strip of 
width $\sigma^m 2^{\abs{\alpha}q_0}\abs{\alpha}^{-m}$. The same holds of course 
for $\eta_\rho(x-x_0)$, for any $x_0\in\T^2$. 
\end{remark}


\subsection{Main results: concentration around stable equilibrium branches}
\label{ssec:stable} 

The main part of our results concern the effect of weak space-time white noise 
on the dynamics near a stable equilibrium branch of the unperturbed equation. 

\begin{assumption}[Stable case]
\label{assump:stable} 
There exists a map $\phi^*:I\to\R$ such that 
\begin{equation}
 F(t,\phi^*(t)) = 0 \; \forall t\in I\;.
\end{equation} 
Furthermore, the linearisation $a(t) = \partial_\phi F(t,\phi^*(t))$ 
satisfies~\eqref{eq:a(t)}. 
\end{assumption}

In the deterministic case $\sigma = 0$, the SPDE~\eqref{eq:SPDE_renorm} reduces 
to 
\begin{equation}
\label{eq:SPDE_det} 
 \6\phi(t,x) = \frac{1}{\eps}\bigbrak{\Delta\phi(t,x) + F(t, \phi(t,x))} \6t\;, 
\end{equation} 
since the renormalisation counterterm $C_N$ vanishes for $\sigma = 0$. 
We then have the following generalisation of Tihonov's theorem 
(cf.~\cite{Tihonov}).

\begin{prop}[Deterministic case]
\label{prop:stable_det} 
There exist constants $\eps_0, C>0$ such that, when $0<\eps<\eps_0$, 
\eqref{eq:SPDE_det} admits a particular solution $\bar\phi(t)$ satisfying   
\begin{equation}
 \norm{\bar\phi(t,\cdot) - \phi^*(t)e_0}_{H^1} \leqs C\eps 
 \qquad \forall t\in I\;.
\end{equation} 
\end{prop}

The difference $\phi_0 = \phi - \bar\phi$ satisfies the SPDE 
\begin{equation}
\label{eq:SPDE_tilde} 
 \6\phi_0(t,x) = \frac{1}{\eps}\bigbrak{\Delta\phi_0(t,x) + 
\Wick{F_0(t,x,\phi_0(t,x))}} \6t + \frac{\sigma}{\sqrt{\eps}} 
\6W(t,x)\;, 
\end{equation} 
where 
\begin{equation}
 \Wick{F_0(t,x,\phi_0(t,x))}
 = \Wick{F(t,\bar\phi(t,x)+\phi_0(t,x))} - F(t,\bar\phi(t,x))
\end{equation} 
has similar properties as $F$, and satisfies in addition 
$F_0(t,x,0) = 0$ for all $t\in I$ and all $x\in\T^2$. More precisely, we 
have the following result. 

\begin{lemma}
\label{lem:F0} 
The renormalised forcing term is given by 
\begin{equation}
\label{eq:Ftilde} 
 \Wick{F_0(t,x,\phi_0(t,x))} 
 = a(t)\phi_0(t,x) + \sum_{j=1}^n \hat A_j(t,x) 
\Wick{\phi_0(t,x)^j}\;, 
\end{equation} 
where the $\hat A_j(t,\cdot)$ belong to $H^1$ (which is embedded in 
$\Besovspace{1}{2}{\infty}$) for all $t\in I$, and are given by 
\begin{equation}
\label{eq:def_Atilde} 
 \hat A_j(t,x) 
 = 
\begin{cases}
\displaystyle
 \sum_{i=2}^n iA_i(t)\bigbrak{\bar\phi(t,x)^{i-1} - 
 \phi^*(t)^{i-1}e_0(x)}\;, 
 & j=1\;, \\[15pt]
\displaystyle
 \sum_{i=j}^n \binom{i}{j} A_i(t) \bar\phi(t,x)^{i-j}\;, 
 & j = 2, \dots, n\;.
\end{cases} 
\end{equation} 
\end{lemma}

\begin{remark}
The proof of~\cite[Theorem~2.4]{Berglund_Nader_22} contains a small mistake, 
which is, however, easily corrected. 
In~\cite[Equation~(3.5)]{Berglund_Nader_22}, $\bar a(t)$ should be defined as 
$\bar a(t) = \partial_\phi f(t, \phi^*(t)e_0)$ instead of 
$\bar a(t) = \partial_\phi f(t, \bar\phi(t,x))$, in order to obtain a value 
independent of $x$. The only change to me made in the proof is that the 
nonlinear term $b$ has order $h^2 + \eps h$ instead of $h^2$. 
\end{remark}

We rewrite~\eqref{eq:SPDE_tilde} as 
\begin{equation}
\label{eq:SPDE_stable} 
 \6\phi_0(t,x) = \frac{1}{\eps}\bigbrak{\Delta\phi_0(t,x) + 
 a(t) \phi_0(t,x) + \Wick{b(t,x,\phi_0(t,x))}} \6t 
 + \frac{\sigma}{\sqrt{\eps}} \6W(t,x)\;, 
\end{equation} 
where $\Wick{b}$ denotes the sum over $j$ in~\eqref{eq:Ftilde}. Note that 
$\Wick{b}$ contains a term linear in $\phi_0$. However, it has a coefficient of 
order $\eps$, since $\bar\phi$ and $\phi^\star$ are at a distance of order 
$\eps$. 

We now apply the Da Prato--Debussche trick, and consider the difference $\phi_1 
= \phi_0 - \psi$. It satisfies the equation 
\begin{equation}
\label{eq:phi1} 
 \6\phi_1(t,x) 
 = \frac{1}{\eps} \bigbrak{\Delta\phi_1(t,x) + a(t)\phi_1(t,x) 
 + \Wick{b(t,x,\psi(t,x)+\phi_1(t,x))}}\6t\;,
\end{equation} 
where
\begin{equation}
 \Wick{b(t,x,\psi(t,x)+\phi_1(t,x))}
 = \sum_{j=1}^n \hat A_j(t,x) 
 \sum_{\ell=0}^j \binom{j}{\ell} \phi_1(t,x)^{j-\ell} \Wick{\psi(t,x)^\ell}\;.
\end{equation} 
It follows from Proposition~\ref{prop:Besov} that if 
$\phi_1\in\Besovspace{\beta}{2}{\infty}$ and 
$\Wick{\psi^\ell}\in\Besovspace{\alpha}{2}{\infty}$ for $\alpha < 0$ and 
$\ell = 0,\dots,n-1$, then 
\begin{equation}
 \Wick{b(t,x,\psi+\phi_1)}
 \in\Besovspace{\bar\alpha}{2}{\infty}
 \qquad 
 \forall \bar\alpha < (2n-1)\alpha\;,
\end{equation} 
provided $\beta\geqs 1+2\alpha$. By the Schauder estimate recalled in 
Proposition~\ref{prop:Schauder}, the solution of~\eqref{eq:phi1} belongs to 
$\Besovspace{\gamma}{2}{\infty}$ for $\gamma < 2 - (2n+1)\abs{\alpha}$, which 
allows to close the fixed-point argument, in accordance with 
Theorem~\ref{thm:daPratoDebussche}.

By the embedding 
$\Besovspace{\gamma}{2}{\infty} \hookrightarrow 
\Besovspace{\gamma-1}{\infty}{\infty} = \cC^{\gamma-1}$, 
we see that the solution of~\eqref{eq:phi1} is H\"older continuous, with 
exponent almost $1$. In other words, the solution is almost Lipschitz 
continuous. Our main result is then the following.

\begin{theorem}[Concentration estimate for $\phi_1$]
\label{thm:phi1} 
For any choice of $\gamma < 2$ and $\nu < 1 - \frac{\gamma}{2}$, there exist 
constants $C(T,\eps), M, \kappa, h_0, \eps_0 > 0$ such that, whenever $\eps < 
\eps_0$ and $h < h_0\eps^\nu$, one has  
\begin{align}
 \biggprob{\sup_{t\in[0,T]} \normBesov{\phi_1(t)}{\gamma}{2}{\infty} > 
M\eps^{-\nu}h(h+\eps)}
 &\leqs C(T,\eps) \e^{-\kappa h^2/\sigma^2}\;, \\
 \biggprob{\sup_{t\in[0,T]} \norm{\phi_1(t)}_{\cC^{\gamma-1}} > 
M\eps^{-\nu}h(h+\eps)}
 &\leqs C(T,\eps) \e^{-\kappa h^2/\sigma^2}\;.
\end{align} 
\end{theorem}

This result shows in particular that sample paths of $\phi_1$ are concentrated 
in a ball in $\cC^{\gamma-1}$-norm of size 
\begin{equation}
 \eps^{-\nu} \sigma (\sigma + \eps) 
 \asymp 
\begin{cases}
  \eps^{-\nu}\sigma^2 & \text{if $\sigma > \eps$\;,} \\
  \eps^{1-\nu}\sigma & \text{if $\sigma \leqs \eps$\;.}
\end{cases}
\end{equation} 


\subsection{The case of bifurcations}
\label{ssec:bif} 

In this section, we comment on how the results of the last section can be 
extended to situations where the nonlinearity $F$ fails to satisfy 
Assumption~\ref{assump:stable}, that is, in the case of a bifurcation. In the 
work~\cite{Berglund_Nader_22}, which concerned SPDEs on the one-dimensional 
torus, we considered the case of an avoided transcritical bifurcation, where 
$F$ is given locally by 
\begin{equation}
 F(t,\phi) = \delta + t^2 - \phi^2 + \bigOrder{(\abs{t} + \abs{\phi})^3}
\end{equation} 
with $0 < \delta \ll 1$. In that case, there is a stable equilibrium branch 
$\phi^*_+(t) \simeq \sqrt{\delta + t^2}$ approaching an unstable branch 
$\phi^*_-(t) \simeq -\sqrt{\delta + t^2}$ at distance $2\sqrt{\delta}$. While 
the linearization $a(t) = \partial_\phi F(t,\phi^*(t))$ remains positive, its 
value becomes small in terms of $\delta$ near $t=0$. As a result, while the 
system still behaves as in the stable case when $\sigma \ll 
(\delta\vee\eps)^{3/4}$, a new behaviour emerges for $\sigma \gg 
(\delta\vee\eps)^{3/4}$: it becomes likely for sample paths to cross the 
unstable equilibrium branch, and travel in a short time to a distant region of 
space. 

Here we will illustrate how these results can be transposed to singular SPDEs 
on the two-dimensional torus. However, for a change, we are going to take as an 
example the equation 
\begin{equation}
\label{eq:SPDE_pitchfork} 
 \6\phi(t,x) = \frac{1}{\eps}\bigbrak{\Delta\phi(t,x) 
 + a(t)\phi(t,x) - \Wick{\phi(t,x)^3}} \6t + \frac{\sigma}{\sqrt{\eps}} 
\6W(t,x)\;,
\end{equation}
which describes a pitchfork bifurcation when $a(t)$ changes from being negative 
to being positive at a time $t^*$. In the deterministic case $\sigma = 0$, there 
is a phenomenon known as bifurcation delay: solutions attracted by the stable 
equilibrium branch $\phi^*(t) = 0$ for $t < t^*$ remain close to $0$ for a time 
of order $1$ beyond the bifurcation time $t^*$, even though the equilibrium 
branch has become unstable. This is due to the solution becoming exponentially 
close to $0$ during the stable phase, and a time of order $1$ being required for 
the solution to reach again values of order $1$.  

In the one-dimensional SDE case, the effect of noise on such a system has been 
studied in~\cite{BG_pitchfork}. The main result of that work is that sample 
paths remain with high probability at a distance of order $\sigma\eps^{-1/4}$ 
from zero up to a time $t^* + \Order{\eps^{1/2}}$, but are unlikely to remain 
close to $0$ after times of order $t^* + 
\Order{(\eps\log(\sigma^{-1}))^{1/2}}$. 
The effect of noise is thus to reduce the bifurcation delay from order $1$ to 
order $(\eps\log(\sigma^{-1}))^{1/2}$.

In order to analyse the SPDE~\eqref{eq:SPDE_pitchfork}, we start by carrying 
out the change of variables 
\begin{equation}
 \phi(t,x) = \psi_\perp(t,x) + \phi_1(t,x)\;,
\end{equation} 
where the stochastic convolution $\psi_\perp$ solves the SPDE 
\begin{equation}
  \6\psi_\perp(t,x) = \frac{1}{\eps}\bigbrak{\Delta_\perp\psi_\perp(t,x) 
 + a(t)\psi_\perp(t,x)} \6t + \frac{\sigma}{\sqrt{\eps}} 
\6W_\perp(t,x)
\end{equation}
with zero initial condition. Here the noise $\6W_\perp$ acts only on non-zero 
Fourier modes, implying that the spatial average of $\psi_\perp(t,x)$ always 
remains equal to zero. We use the notation $\Delta_\perp$ to emphasize that the 
Laplacian only acts on non-zero Fourier modes, although it has the same effect 
as the usual Laplacian. The resulting equation for $\phi_1$ reads 
\begin{equation}
\label{eq:SPDE_pitch_phi1} 
 \6\phi_1(t,x) = \frac{1}{\eps}\bigbrak{\Delta\phi_1(t,x) 
 + a(t)\phi_1(t,x) + \Wick{F(\psi_\perp(t,x), \phi_1(t,x))}} \6t + 
\frac{\sigma}{\sqrt{\eps}} \6W_0(t,x)\;,
\end{equation}
where
\begin{equation}
\Wick{F(\psi_\perp, \phi_1)}
 = - \Wick{\psi_\perp^3} - 3\phi_1\Wick{\psi_\perp^2}
 - 3\phi_1^2\psi_\perp - \phi_1^3\;.
\end{equation} 
The next step is to split $\phi_1$ into its mean and oscillating spatial part, 
writing 
\begin{equation}
 \phi_1(t,x) = \phi_1^0(t)e_0(x) + \phi_1^\perp(t,x)\;, \qquad 
 \phi_1^0(t) = \pscal{e_0}{\phi_1(t,\cdot)}\;.
\end{equation} 
This results in the coupled SDE--SPDE system 
\begin{align}
\label{eq:bif_phi10} 
  \6\phi_1^0(t) &= \frac{1}{\eps}\bigbrak{ 
 a(t)\phi_1^0(t) - \phi_1^0(t)^3 + F_0(\psi_\perp, \phi_1^0, \phi_1^\perp)} 
\6t + 
\frac{\sigma}{\sqrt{\eps}} \6W_0(t)\;,\\
 \6\phi_1^\perp(t,x) &= \frac{1}{\eps}\bigbrak{\Delta_\perp\phi_1^\perp(t,x) 
 + a(t)\phi_1^\perp(t,x) + \Wick{F_\perp(\psi_\perp, \phi_1^0, \phi_1^\perp)}} 
\6t\;,
\label{eq:bif_phi1perp} 
\end{align}
where $F_0$ and $\Wick{F_\perp}$ are nonlocal nonlinearities given by 
\begin{align}
 F_0(\psi_\perp, \phi_1^0, \phi_1^\perp) 
 &= \pscal{e_0}{\Wick{F(\psi_\perp, \phi_1^0 e_0 + \phi_1^\perp)}}\;, \\
 &= \delta_0(\Wick{F(\psi_\perp, \phi_1^0 e_0 + \phi_1^\perp)})\;, \\
 F_\perp(\psi_\perp, \phi_1^0, \phi_1^\perp) 
 &= \Wick{F(\psi_\perp, \phi_1^0 e_0 + \phi_1^\perp)} 
 - F_0(\psi_\perp, \phi_1^0, \phi_1^\perp)\;.
\end{align}
We start by describing concentration properties of $\phi_1^\perp$. For that, 
given a parameter $H_0 > 0$, we introduce the stopping time 
\begin{equation}
 \tau_0(H_0) = \inf\bigset{t\in[0,T] \colon \abs{\phi_1^0(t)} > H_0}\;.
\end{equation} 

\begin{theorem}[Concentration estimate for $\phi_1^\perp$]
\label{thm:phi1perp} 
Assume there exists a constant $a_0 > 0$ such that $a(t) \leqs (2\pi)^2 - a_0$ 
for all $t\in[0,T]$. Then for any choice of $\gamma < 2$ and $\nu < 1 - 
\frac\gamma2$, there exist constants $C(T,\eps), M, \kappa, h_0 > 0$ such that, 
whenever $h + H_0 \leqs h_0 \eps^{\nu/2}$, one has 
\begin{equation}
 \biggprob{\sup_{t\in[0,T\wedge\tau_0(H_0)]} 
 \smallnorm{\phi_1^\perp(t)}_{\cC^{\gamma-1}} > M\eps^{-\nu}(h + H_0)^3}
 \leqs C(T,\eps) \e^{-\kappa h^2/\sigma^2}\;.
\end{equation} 
\end{theorem}

Note in particular the weaker condition on $a(t)$: instead of having to stay 
negative, $a(t)$ may become positive, as long as it stays smaller than 
$(2\pi)^2$. This is because the eigenvalues of the Laplacian $\Delta_\perp$ are 
bounded above by $-(2\pi)^2$. 

\begin{remark}
One can easily extend the result to cases where 
$a(t)$ exceeds the value $(2\pi)^2$ by incorporating more Fourier modes in the 
variables $\phi_1^0$. 
\end{remark}

It is now relatively straightforward to extend the one-dimensional results 
from~\cite{BG_pitchfork} to the SDE~\eqref{eq:bif_phi10} governing the zeroth 
Fourier mode.
The idea is that its solution is likely to remain close, on some time interval, 
to the solution of the linearised equation
\begin{equation}
\label{eq:phicirc} 
 \6\phi^\circ(t) = \frac{1}{\eps}\
 a(t)\phi^\circ(t)\6t + \frac{\sigma}{\sqrt{\eps}} \6W_0(t)\;,
\end{equation} 
which is a Gaussian process, with variance 
\begin{equation}
 v^\circ(t) = v^\circ(0) + \frac{\sigma^2}{\eps}
 \int_0^t \e^{2\alpha(t,t_1)/\eps} \6t_1\;, 
 \qquad
 \alpha(t,t_1) = \int_{t_1}^t a(t,t_2)\6t_2\;.
\end{equation} 
One can show (see~\cite[Lemma~4.2]{BG_pitchfork}) that for an initial variance 
$v^\circ(0)$ of order $\sigma^2$, bounded away from zero, one has 
\begin{equation}
 v^\circ(t)
 \asymp 
 \begin{cases}
  \dfrac{\sigma^2}{\abs{t - t^*}}
  & \text{for $0\leqs t\leqs t^* - \sqrt{\eps}$\;,} \\[10pt]
  \dfrac{\sigma^2}{\sqrt{\eps}}
  & \text{for $- \sqrt{\eps} \leqs t-t^* \leqs \sqrt{\eps}$\;,} \\[10pt]
  \dfrac{\sigma^2}{\sqrt{\eps}} \e^{2\alpha(t,t^*)/\eps}
  & \text{for $t\geqs t^* + \sqrt{\eps}$\;.} 
 \end{cases}
\end{equation}
Note that the variance increases slowly up to time $t^*+\sqrt{\eps}$, and then 
increases exponentially fast. This suggests defining sets 
\begin{align}
 \cB_-(h_-) &= \biggset{(t,\phi_1^0) \in [0, t^*+\sqrt{\eps}]\times\R \colon 
 \abs{\phi_1^0} \leqs \frac{h_-}{\sigma} v^\circ(t)}\;,\\
 \cB_+(h_+) &= \biggset{(t,\phi_1^0) \in [t^*+\sqrt{\eps}, T]\times\R \colon 
 \abs{\phi_1^0} \leqs \frac{h_+}{\sqrt{a(t)}}}\;.
\end{align} 
The first set is a union of confidence intervals associated with the variance 
$v^\circ(t)$. The second set is motivated by the fact that the variance of 
processes obtained by linearization grows like 
\begin{equation}
\frac{\sigma^2}{2a(t)}\e^{2\alpha(t,t^*)/\eps}\;. 
\end{equation} 
One then has the following generalisation of~\cite[Theorem~2.10]{BG_pitchfork} 
and~\cite[Proposition~4.7]{BG_pitchfork}. 

\begin{theorem}[Behaviour of $\phi_1^0(t)$ near a pitchfork bifurcation]
\label{thm:phi10} 
There exist positive constants $M, \eps_0, h_0$ such that, for any 
$\eps < \eps_0$ and $h_- \leqs h_0\eps^{1/2}$, and any $t \leqs t^* + 
\eps^{1/2}$, one has 
\begin{equation}
\label{eq:pitchfork_bound1} 
 \bigprob{\tau_{\cB_-(h_-)} \leqs t}
 \leqs C(t,\eps) \exp\biggset{-\frac{h_-^2}{2\sigma^2} \biggbrak{1 - 
\Order{\sqrt{\eps}\,} - \biggOrder{\frac{h_-^2}{\eps}}}}\;,
\end{equation} 
where $C(t,\eps) = \Order{\alpha(t)/\eps^2}$. 
Furthermore, for $h_+ = \sigma\log(\sigma^{-1})^{1/2}$ and 
any $t \geqs t^* + \eps^{1/2}$, one has 
\begin{equation}
\label{eq:pitchfork_bound2} 
 \bigprob{\tau_{\cB_+(h_+)} \geqs t}
 \leqs \frac{h_+}{\sigma} \exp\biggset{-\kappa \frac{\alpha(t,t^*)}{\eps}}
 + C(t,\eps) \e^{-\kappa \log(\sigma^{-1})/\sqrt{\eps}}
\end{equation} 
for a constant $\kappa > 0$. 
\end{theorem}

The bound~\eqref{eq:pitchfork_bound1} shows that when $\sigma \ll h_- 
\ll\sqrt{\eps}$, sample paths are likely to stay in $\cB_-(h_-)$ up to time 
$t^*+\sqrt{\eps}$. At time $t^*+\sqrt{\eps}$, typical fluctuations have a size 
of order $\sigma\eps^{-1/4}$. Since $\alpha(t,t^*)$ grows like $(t-t^*)^2$, the 
bound~\eqref{eq:pitchfork_bound2} shows that sample paths are likely to leave a 
neighborhood of size $\sigma$ of $0$ at times of order 
$\sqrt{\eps\log(\sigma^{-1})}$. 


\section{Proof of Theorem~\ref{thm:stoch_convolution}}
\label{sec:proof_stoch_convolution}

\subsection{Relation between Hermite polynomials and binomial formula}

We will first give a proof of Theorem~\ref{thm:stoch_convolution} in the 
particular case where $a(t) = -1$ for all $t$. That is, we consider the linear 
equation 
\begin{equation}
 \6\psi(t,x) = \frac{1}{\eps}\bigbrak{\Delta\psi(t,x) - \psi(t,x)} \6t 
 + \frac{\sigma}{\sqrt{\eps}} \6W(t,x)\;.
\end{equation} 
Its projection on the $k$th basis vector $e_k$ is given by
\begin{equation}
\label{eq:stoch_convolution_k} 
 \6\psi_k(t) = -\frac{1}{\eps}(\mu_k+1)\psi_k(t) \6t 
 + \frac{\sigma}{\sqrt{\eps}} \6W_k(t)\;,
\end{equation} 
where the $\mu_k$ are the eigenvalues~\eqref{eq:ev_Laplacian} of the Laplacian, 
and the $\set{W_k(t)}_{t\geqs0}$ are independent Wiener processes (actually, 
since we use complex Fourier series, we have $\expec{\6W_{k_1}\6W_{k_2}} = 
\delta_{k_1,-k_2}\6t$, but this yields equivalent results). We write 
$\alpha_k(t,t_1) = -(\mu_k+1)(t-t_1)$ and $\alpha_k(t,0) =  \alpha_k(t)$ for 
brevity. The solution of \eqref{eq:stoch_convolution_k} is an 
Ornstein--Uhlenbeck process, which can be represented using Duhamel's principle 
by the Ito integral
\begin{equation}
\label{stoch_convolution_k} 
 \psi_k(t) =  \e^{\alpha_k(t)/\eps}\psi_k(0) + 
\frac{\sigma}{\sqrt{\eps}}\int_0^t \e^{\alpha_k(t,t_1)/\eps} \6W_k(t_1)\;.
\end{equation} 
At any time $t\geqs0$, $\psi_k(t)$ is a zero-mean Gaussian 
random variable of variance 
\begin{align}
\label{var:stoch_convolution_k} 
v_k(t) &= \variance\bigpar{\psi_k(0)}\e^{2\alpha_k(t)/\eps} + 
\frac{\sigma^2}{\eps}\int_0^t \e^{2\alpha_k(t,t_1)/\eps} \6t_1 \\
&= \variance\bigpar{\psi_k(0)}\e^{2\alpha_k(t)/\eps} +
\frac{\sigma^2}{2(\mu_k+1)} \Bigbrak{1 - \e^{2\alpha_k(t)/\eps}}\;.
\end{align} 
In order to obtain a stationary process, we will assume that the initial 
conditions $\psi_k(0)$ follow centred normal distributions with variance 
$v_k = \sigma^2/[2(\mu_k+1)]$, which are mutually independent, and independent 
of the Wiener processes. In this way, we have $v_k(t) = v_k$ for all $t$. 

Fix $\alpha<0$. By Definition \ref{norm_besov} of Besov norms, 
\begin{align}
P_m(h) :={}&\Bigprob{\sup_{0\leqs t\leqs T}
\norm{\Wick{\psi(t,\cdot)^m}}_{\cB^\alpha_{2,\infty}}>h^m} \\
={}&\Bigprob{\sup_{0\leqs t\leqs T}\sup_{q_0\geqs0}2^{-\abs{\alpha}q_0}
\norm{\delta_{q_0}(\Wick{\psi(t,\cdot)^m})}_{L^2}>h^m}\;\\
={}&\Bigprob{\exists q_0\geqs0; \sup_{0\leqs t\leqs T}
\norm{\delta_{q_0}(\Wick{\psi(t,\cdot)^m})}_{L^2}>h^m2^{\abs{\alpha}q_0 }}\;\\
\leqs{}&\sum_{q_0\geqs0}\Bigprob{\sup_{0\leqs t\leqs T}
\norm{\delta_{q_0}(\Wick{\psi(t,\cdot)^m})}_{L^2}>h^m2^{\abs{\alpha}q_0}}\;.
\label{eq:bound_Besovnorm_1} 
\end{align}

\begin{remark}
At any fixed time $t$, the law of $\psi(t,\cdot)$ is that of the truncated 
Gaussian free field, with variance 
\begin{equation}
 \sum_{k\in\Z^2 \colon \abs{k}\leqs N} v_k = C_N\;,
\end{equation} 
which diverges like $\sigma^2\log(N)$, as mentioned in \eqref{eq:def_CN}.
\end{remark}

In what follows, it will be convenient to use multiindex notations. 
For any $\hbn\in\N^{\N}$ with finitely many nonzero components, we write 
\begin{equation}
\abs{\hbn}=\sum_{q\geqs0}\hbn_q \qquad\text{and}\qquad  
\hbn!:=\prod_{q\geqs0}\hbn_q!\;.
\end{equation}
Since $\hbn$ has finitely many nonzero components, these quantities are 
indeed well-defined. Let 
\begin{equation}
\brak{\hbn}=\#\set{q:\hbn_q>0}
\end{equation}
be the number of these nonzero components. We can order them as 
$q_1<q_2<...<\q$, where 
\begin{equation}
\q=\max\{q:\hbn_q>0\}
\end{equation}
is the index of the largest nonzero entry of $\hbn$. In what follows, we will 
always assume that $\abs{\hbn} = m$. We notice that this implies $\brak{\hbn} 
\leqs m$. 

The projection of $\psi(t,\cdot)$ on the annulus $\cA_q$ has constant variance 
\begin{equation}
 c_q := \E\Bigbrak{\norm{\delta_q\psi(t,\cdot)}^2_{L^2}}
 = \sum_{k\in\cA_q} v_k\;. 
\end{equation} 
In fact, $\delta_q\psi(t,x)$ has variance $c_q$ for all $x\in\T^2$. 
An important feature of the projections is that 
\begin{equation}
 c_q \lesssim \sum_{k\in\cA_q} \frac{\sigma^2}{1+\norm{k}^2}
 \lesssim 2\sigma^2\int_{2^{q-1}}^{2^q} \frac{r\6r}{1+r^2}
 = \sigma^2\log\biggpar{\frac{1+2^{2q}}{1+2^{2(q-1)}}} \leqs 2\sigma^2\log2
\end{equation} 
for all $q$. Finally note that the cut-off condition $\abs{k}\leqs N$ implies 
$q\leqs \ell_N = \lfloor \log_2N \rfloor$, and that 
\begin{equation}
 C_N = \sum_{q=0}^{\ell_N}c_q\;.
\end{equation}
With these notations in place, we can introduce the binomial formula for 
Hermite polynomials, see Lemma~\ref{lem:hermite_multinomial} in 
Appendix~\ref{app:Wick}. 

\begin{lemma}
\label{binomial_formula} 
For any $m\in\N$, we have 
\begin{equation}
H_m(\psi(t,\cdot); C_N)
=\sum_{\abs{\hbn}=m}\frac{m!}{\hbn!}\prod_{q\geqs0}H_{\hbn_q}(\delta_q\psi(t,
\cdot),c_q)\;.
\end{equation} 
\end{lemma}

It thus follows from~\eqref{eq:bound_Besovnorm_1} that we have 
\begin{align}
P_m(h)
\leqs{}&\sum_{q_0\geqs0}\Bigprob{\sup_{0\leqs t\leqs T}\bignorm{\delta_{q_0}(\sum_{\abs{\hbn}=m}\frac{m!}{\hbn!}\prod_{q\geqs0}\Wick{\delta_q\psi(t,\cdot)^{\hbn_q}})}_{L^2}>h^m2^{\abs{\alpha} q_0}}\;\\
={}&\sum_{q_0\geqs0}\Bigprob{\sup_{0\leqs t\leqs T}\bignorm{\sum_{\abs{\hbn}=m}\frac{m!}{\hbn!}\delta_{q_0}(\prod_{q\geqs0}\Wick{\delta_q\psi(t,\cdot)^{\hbn_q}})}_{L^2}>h^m2^{\abs{\alpha} q_0}}\;\\
\leqs{}&\sum_{q_0\geqs0}\Bigprob{\sup_{0\leqs t\leqs T}\sum_{\abs{\hbn}=m}\frac{m!}{\hbn!}\bignorm{\delta_{q_0}(\prod_{q\geqs0}\Wick{\delta_q\psi(t,\cdot)^{\hbn_q}})}_{L^2}>h^m2^{\abs{\alpha} q_0}}\;.
\end{align}

\begin{remark}
\label{rem:q0} 
Note that for any $q_1,\ q_2\geqs0$, one has $2^{q_1}+2^{q_2}\leqs 
2^{\max\{q_1,q_2\}+1}$. Therefore, 
\begin{equation}
\delta_{q_0}(\prod_{q\geqs0}\Wick{\delta_q\psi(t,\cdot)^{\hbn_q}})\ne0 
\qquad \Rightarrow \qquad 
q_0\leqs \max_{i\leqs\q}\set{q_i+\hbn_{q_i}}
\leqs \q+\hbn_{\q}
\end{equation}
for any $\hbn$, which will be useful in restricting the domains of the sums. 
\end{remark}

For any decomposition $h^m=\displaystyle\sum_{\abs{\hbn}=m}h_\hbn^m$, one has
\begin{equation}
P_m(h) 
\leqs \sum_{q_0\geqs0}\sum_{\abs{\hbn}=m}
\Bigprob{\sup_{0\leqs t\leqs T}
\bignorm{\delta_{q_0}(\prod_{q\geqs0}
\Wick{\delta_q\psi(t,\cdot)^{\hbn_q}})}_{L^2}>\frac{\hbn!}{m!}h_\hbn^m2^{\abs{
\alpha} q_0}}\;.
\label{eq:bound_Besovnorm_2} 
\end{equation}


\subsection{Martingale and partition}
\label{ssec:martingale}

In this section, we fix $q_0\geqs0$ and $\smash{\hbn\in\N^\N}$ with $\abs{\hbn} 
= m$. Our aim is to estimate one term in the double 
sum~\eqref{eq:bound_Besovnorm_2}. We notice that the stochastic integral 
$\psi_k(t)$ is not a martingale. However, 
\begin{equation}
\label{eq:psi_hat}
\e^{-\alpha_k(t)/\eps}\psi_k(t) =   \psi_k(0) + 
\frac{\sigma}{\sqrt{\eps}}\int_0^t \e^{-\alpha_k(t_1)/\eps} \6W_k(t_1)
\end{equation} 
is a martingale of variance $\e^{-2\alpha_k(t)/\eps}v_k$. 
The variances of $\psi_k(t)$ and $\hat\psi_k(t)$ are too different on the whole 
time interval $[0,T]$ to allow a useful comparison of the two processes. This 
is why we introduce a partition $0=u_0\leqs u_1<\cdots<u_L=T$ of this 
interval. Given $\gamma_0>0$ and any $k_0\in\Z^2$ such that $\abs{k_0}=2^{\q}$, 
we define the partition by 
\begin{equation}
\label{eq:def_partition} 
 \alpha_{k_0}(u_{l+1},u_l)=-\gamma_0\eps
 \qquad \text{for }
 1\leqs l\leqs L =
\biggl\lfloor\frac{((2\pi)^2\norm{k_0}^2 + 1)T}{\gamma_0\eps}
\biggr\rfloor\;,
\end{equation} 
and write $I_l = [u_l, u_{l+1}]$. 
Multiplying~\eqref{eq:psi_hat} by $\e^{\alpha_k(u_{l+1})/\eps}$,  
we obtain the martingale
\begin{equation}
\label{eq:psi_k_hat} 
\hat\psi_k(t)
:= \e^{\alpha_k(u_{l+1},t)/\eps} \psi_k(t) =   
\e^{\alpha_k(u_{l+1})/\eps}\psi_k(0) + \frac{\sigma}{\sqrt{\eps}}\int_0^t 
\e^{\alpha_k(u_{l+1},t_1)/\eps} \6W_k(t_1)\;,
\end{equation} 
where we do not indicate the $l$-dependence of $\hat\psi_k(t)$ in order not 
to overload the notations. The variance of $\hat\psi_k(t)$ is 
\begin{equation}
\label{eq:v_k_hat} 
 \hat v_k(t) = v_k \e^{2\alpha_k(u_{l+1},t)/\eps}\;.
\end{equation}
They key observation is the following property of Hermite polynomials, which is 
proved in Appendix~\ref{app:Wick}. 

\begin{lemma}
\label{lem:Hermite_martingale} 
For any $m\geqs1$, $\set{H_m(\hat\psi_k(t); \hat v_k(t))}_{t\geqs0}$ is a 
martingale with respect to the canonical filtration $\set{\cF_t}_t$ of 
$(W_k(t))_t$.
\end{lemma}

This observation will allow us to deal with the supremum over times 
in~\eqref{eq:bound_Besovnorm_2}, by using Doob's submartingale inequality. We 
will thus be interested in the martingales 
\begin{equation}
 \delta_q \hat\psi(t,x) 
 = \sum_{k\in\cA_q} \hat\psi_k(t) e_k(x)\;,
\end{equation} 
as well as of the related quantities 
\begin{equation}
 X_{\hbn}^2(t) 
= \bignorm{\delta_{q_0}(\prod_{q\geqs0}
\Wick{\delta_q\hat\psi(t,\cdot)^{\hbn_q}})}^2_{L^2}\;.
\end{equation} 
Later on, we will extend the obtained bounds to functions of $\delta_q 
\psi(t,x)$. 

\begin {prop}
\label{prop:sup_Xn} 
Fix a constant $\gamma_{q_0,\q}\in\R$ and $l\in\set{0,\dots,L}$. Then the bound
\begin{equation}
\biggprob{\sup_{t\in I_l}
X_{\hbn}^2(t)
> \Bigpar{\frac{\hbn!}{m!}h_\hbn^m2^{\abs{\alpha} q_0}}^2}
\leqs C_{\hbn}(l,\eps)\exp\biggset{-\frac{H_{\hbn}(q_0,l)}{\sigma^2}}
\end{equation}
holds, where
\begin{align}
\label{eq:martingale_bounds} 
C_{\hbn}(l, \eps)
={}& \e^{m-1} {}+{} 
\biggexpec{\exp\biggset{\frac{\gamma_{q_0,\q}}{\sigma^2} 
\bigbrak{X_{\hbn}^2(u_{l+1})}^{1/m}}}\;,\\
H_{\hbn}(q_0,l)
={}& \gamma_{q_0,\q}
\biggpar{\frac{\hbn!}{m!}2^{\abs{\alpha} q_0}h_\hbn^m}^{2/m}\;.
\end{align}
\end{prop}
\begin {proof}
The process $(X_{\hbn}^2(t))_{t\in I_l}$ is a submartingale, because it is the 
projection of a sum of squares of independent martingales. We note that the 
function $f_\gamma:\R_+\to\R_+$ given by 
\begin{equation}
 f_\gamma(x) = \max\bigset{\e^{m-1}, \e^{\gamma x^{1/m}}}
 = 
 \begin{cases}
 \e^{m-1} & \text{if $x \leqs \biggpar{\dfrac{m-1}{\gamma}}^m$\;,} \\[10pt]
 \e^{\gamma x^{1/m}} &\text{if $x > \biggpar{\dfrac{m-1}{\gamma}}^m$}
 \end{cases}
\end{equation} 
is non-decreasing and convex. 
By Doob's submartingale inequality, we get
\begin{align}
\biggprob{\sup_{t\in I_l}
X_{\hbn}^2(t)
> \Bigpar{\frac{\hbn!}{m!}h_\hbn^m2^{\abs{\alpha} q_0}}^2}
&= \biggprob{\sup_{t\in I_l} f_\gamma\bigpar{X_{\hbn}^2(t)}
>f_\gamma\biggpar{\Bigpar{\frac{\hbn!}{m!}h_\hbn^m2^{\abs{\alpha} q_0}}^2}} \\
&\leqs \frac{1}{f_\gamma\Bigpar{\bigpar{\frac{\hbn!}{m!}h_\hbn^m2^{\abs{\alpha} 
q_0}}^2}}
\Bigexpec{f_\gamma\bigpar{X_{\hbn}^2(u_{l+1})}}\;.
\end{align}
In the denominator, we bound $f_\gamma(x)$ below by $\e^{\gamma x^{1/m}}$. 
In the expectation, we bound the maximum defining $f_\gamma$ above by the sum. 
Setting $\gamma = \gamma_{q_0,\q}/\sigma^2$ yields the result. 
\end{proof}

\begin{remark}
Corollary 6.13 in \cite{Janson_book} implies that if $X$ is a polynomial of 
degree $m$ in the field, then $\E\bigbrak{\e^{t\abs{X}}}$ is finite for $m=2$,  
and is in general infinite if $m\geqs3$. This explains the $m$th 
root in Proposition~\ref{prop:sup_Xn}.
\end{remark}

In order to bound the exponential moment in~\eqref{eq:martingale_bounds}, we 
provide the following technical lemma, whose proof is postponed to 
Appendix~\ref{appendix_B}.

\begin {lemma}
\label{lemma:Xq4}
There exists a numerical constant $C_0\geqs0$ such that for any $l$, one has
\begin{equation}
\label{bound_Xq2}
\E\bigbrak{X_{\hbn}^2(u_{l+1})} 
\leqs C_m\sigma^{2m}\frac{2^{2q_0}}{2^{2\q}} 
\end{equation}
where $C_m = C_0^m m!$. 
\end{lemma}

The bound~\eqref{bound_Xq2} says that although high frequency modes, of order 
$2^{\q}$, have some influence on lower modes of order $2^{q_0}$, this influence 
decreases exponentially in their ratio. 

\begin{prop}
\label{bound_expec}
For any $\gamma_{q_0,\q} < \e^{-1}(C_m)^{-1/m}2^{2(\q-q_0)/m}$, one has
\begin{equation}
\biggexpec{\exp\biggset{\frac{\gamma_{q_0,\q}}{\sigma^2} 
\bigbrak{X_{\hbn}^2(u_{l+1})}^{1/m}}}
<\frac{1}{1-\gamma_{q_0,\q} \e C_m^{1/m}2^{2(q_0-\q)/m}}\;.
\end{equation}
\end{prop}

\begin {proof}
Expanding the exponential, we get
\begin{equation}
\biggexpec{\exp\biggset{\frac{\gamma_{q_0,\q}}{\sigma^2} 
\bigbrak{X_{\hbn}^2(u_{l+1})}^{1/m}}}
=\sum_{p\geqs0}\frac{\gamma_{q_0,\q}^p}{\sigma^{2p}p!}
\E\Bigbrak{\bigpar{X_{\hbn}^2(u_{l+1})}^{p/m}}\;.
\end{equation}
By Jensen's inequality (or H\"older's inequality with conjugates $m$ and 
$\frac{m}{m-1}$), we have
\begin{equation}
\E\Bigbrak{\bigpar{X_{\hbn}^2(u_{l+1})}^{p/m}}
\leqs \E\Bigbrak{\bigpar{X_{\hbn}^2(u_{l+1})}^p}^{1/m}\;.
\end{equation}
Since $X_{\hbn}^2(u_{l+1})$ belongs to the $2m$th Wiener chaos, we can use for 
even $p$ equivalence of norms (see Lemma~\ref{lemma:Nelson}) to obtain the 
bound 
\begin{equation}
 \E\Bigbrak{\bigpar{X_{\hbn}^2(u_{l+1})}^p}
 \leqs (p-1)^{mp} \bigexpec{X_{\hbn}^2(u_{l+1})}^p 
 \leqs \Bigbrak{(p-1)^m C_m\sigma^{2m}2^{2(q_0-\q)}}^p\;,
\end{equation} 
where we have used Lemma~\ref{lemma:Xq4} in the last inequality. 
A similar bound follows for odd $p$ by the Cauchy--Schwarz inequality. 
Combining these bounds, we get
\begin{equation}
\biggexpec{\exp\biggset{\frac{\gamma_{q_0,\q}}{\sigma^2} 
\bigbrak{X_{\hbn}^2(u_{l+1})}^{1/m}}}
\leqs \sum_{p\geqs0} \frac{(p-1)^p}{p!}
\Bigbrak{\gamma_{q_0,\q}C_m^{1/m}2^{2(q_0-\q)/m}}^p\;.
\end{equation}
Stirling's formula yields $p^p/p! \leqs \e^p$. 
The result follows by summing a geometric series. 
\end{proof}

Choosing $\gamma_{q_0,\q}=(2\e C_m^{1/m})^{-1}2^{2(\q-q_0)/m}$, we obtain 
\begin{align}
 C_{\hbn}(l, \eps) &\leqs 2 + \e^{m-1}\;, \\
 H_{\hbn}(q_0,l) &= \frac{1}{2\e C_m^{1/m}} 
 \biggpar{\frac{\hbn!}{m!}2^{\abs{\alpha} q_0}2^{\q-q_0}h_\hbn^m}^{2/m}\;.
\end{align} 
This motivates the choice
\begin{equation}
h_\hbn^{m} = \frac{1}{K_m(q_0)} h^m \frac{m!}{\hbn!}\frac{1}{2^{(\q-q_0)/2}}
\indicator{\q + \hbn_{\q} \geqs q_0}\;,
\end{equation}
where the indicator is due to Remark~\ref{rem:q0}, which yields
\begin{equation}
\label{eq:bound_Il_1} 
 \biggprob{\sup_{t\in I_l}
X_{\hbn}^2(t)
> \Bigpar{\frac{\hbn!}{m!}h_\hbn^m2^{\abs{\alpha} q_0}}^2}
\leqs (2+\e^{m-1}) 
\exp\biggset{-\frac{h^2}{2\e\sigma^2} 
\biggpar{\frac{2^{(\q-q_0)/2} 
2^{\abs{\alpha}q_0}}{K_m(q_0)C_m^{1/2}}}^{2/m}}\;.
\end{equation}
The condition $h^m=\displaystyle\sum_{\abs{\hbn}=m}h_\hbn^m$ imposes 
\begin{equation}
\label{eq:Kmq0} 
K_m(q_0) = 
\sum_{\substack{\abs{\hbn}=m \\ \q + \hbn_{\q} \geqs q_0}}
\frac{m!}{\hbn!}
\frac{1}{2^{(\q-q_0)/2}}\;.
\end{equation}
The proof of the following bound is postponed to Appendix~\ref{appendix_B}. 

\begin{lemma}
\label{somme_ngras}
There exist numerical constants $c_0, c_1, c_2 > 0$ such that 
\begin{equation}
K_m(q_0)
\leqs c_0 m! (m + c_2)^m(q_0 + c_1)^m\;.
\end{equation}
\end{lemma}

Substituting in~\eqref{eq:bound_Il_1} yields the bound 
\begin{equation}
\label{eq:bound_Il_2} 
 \biggprob{\sup_{t\in I_l}
X_{\hbn}^2(t)
> \Bigpar{\frac{h^m 2^{\abs{\alpha} q_0}}{K_m(q_0) 2^{(\q - q_0)/2}}}^2}
\leqs (2+\e^{m-1}) 
\exp\biggset{-\kappa_m
\frac{2^{(\q-q_0)/m} 2^{2\abs{\alpha}q_0/m}}
{(q_0+c_1)^2}\frac{h^2}{\sigma^2} }\;,
\end{equation}
where
\begin{equation}
 \kappa_m = \frac{1}{2\e c_0^{2/m}(C_m(m!)^2)^{1/m} (m+c_2)^2}
 = \biggOrder{\frac{1}{m^5}}\;.
\end{equation} 
We now have to convert the estimate~\eqref{eq:bound_Il_2} into an estimate 
involving Wick powers of $\delta_q\psi(t,\cdot)$ instead of 
$\delta_q\hat\psi(t,\cdot)$. For that, we are going to use the following, 
rather rough bound. 
For any $\hbl\in\N^\N$ with finitely many nonzero components, we write 
\begin{equation}
\abs{\hbl} := \sum_{q\geqs0}\hbl_q\;, \qquad 
\hbl!:=\prod_{q\geqs0}\hbl_q!\;,
\qquad\text{and}\qquad  
\hbl \leqs \lfloor \tfrac{\hbn}2 \rfloor 
\quad\Leftrightarrow\quad 
\hbl_q \leqs \lfloor \tfrac{\hbn_q}2 \rfloor \;\forall q\geqs0\;.
\end{equation}
We introduce the shorthands 
\begin{align}
\varphi(t,\cdot) 
&= \prod_{q\geqs0}H_{\hbn_q}(\delta_q\psi(t,\cdot); c_q)\;, \\
\hat\varphi(t,\cdot)  
&= \prod_{q\geqs0}H_{\hbn_q}(\delta_q\hat\psi(t,\cdot); \hat c_q(t))\;.
\end{align}
The proof of the following result is postponed to Appendix~\ref{appendix_B}. 

\begin{prop}
\label{prop:phi_hatphi} 
There is a numerical constant $c_1$ such that for all $t\in I_l$, one has 
\begin{multline}
\bignorm{\delta_{q_0}\bigpar{\varphi(t,\cdot)  
- \hat\varphi(t,\cdot)}}_{L^2}\;\\
\leqs 2^{q_0} \bigpar{c_1\gamma_0}^{[\hbn]} 
\Biggpar{\prod_{\substack{q\geqs0\\\hbn_q>0}} \hbn_q}
\sum_{\hbl:\hbl \leqs\lfloor\hbn/2\rfloor}A_{\hbn\hbl}\hat 
c(t)^{\abs{\hbl}}
\prod_{q\geqs0}\bignorm{\delta_q\hat\psi(t,\cdot)}^{\hbn_q-2\hbl_q}_{L^2}\;,
\qquad 
\label{eq:bound_norm_difference} 
\end{multline}
where 
\begin{equation}
\hat c(t) = \sup_{q\geqs0} \hat c_q(t)
\qquad\text{and}\qquad 
A_{\hbn\hbl} = \frac{\hbn!}{2^{\abs{\hbl}} \hbl! (\hbn-2\hbl)!}
2^{2\q(\abs{\hbn} - 2\abs{\hbl})}\;.
\end{equation}
Note that the first product over $q$ in~\eqref{eq:bound_norm_difference} can be 
bounded above by $m^{[\hbn]}$. 
\end{prop}

We can now derive the main estimate of this section.

\begin{prop}
\label{prop:delta_q0_phi_main} 
There is a constant $Q_m = \Order{m^{-1/2}}$ such that, if one chooses  
$\gamma_0$ of order $q_0^m Q_m 2^{-(m+1)\q}$, there exists a constant 
$\bar\kappa_m$, comparable to $\kappa_m$, such that 
\begin{equation}
 \biggprob{\sup_{t\in I_\ell}
 \norm{\delta_{q_0}\ph(t,\cdot)}_{L^2} > 
 \frac{2^{\abs{\alpha}q_0}h^m}{K_m(q_0) 2^{(\q - q_0)/2}}}
 \leqs \bar C_m(\hbn)
 \exp\biggset{-\bar\kappa_m
\frac{2^{(\q-q_0)/m} 2^{2\abs{\alpha}q_0/m}}
{(q_0+c_1)^2}\frac{h^2}{\sigma^2} }
\end{equation} 
holds for all $h\geqs\sigma$, where
$\bar C_m(\hbn) = 2+\e^{m-1} + c_0(\q + m)$ for a numerical constant $c_0$.
\end{prop}
\begin{proof}
The argument is essentially deterministic. We introduce the two events 
\begin{align}
 \Omega_1(\tilde h) 
 &= \biggset{\forall q \leqs \q + \hbn_{\q},\; \sup_{t\in I_\ell} 
 \smallnorm{\delta_q\hat\psi(t,\cdot)}_{L^2} \leqs \tilde h}\;, \\
 \Omega_2(h, q_0) 
 &= \biggset{\sup_{t\in I_\ell} \norm{\delta_{q_0}\hat\ph(t,\cdot)}_{L^2}
 \leqs \frac12 \frac{2^{\abs{\alpha}q_0}h^m}{K_m(q_0) 2^{(\q - q_0)/2}}}\;.
\end{align}
The estimate~\eqref{eq:bound_Il_2} provides an upper bound on 
$\fP(\Omega_2(h, q_0)^c)$. As for $\Omega_1(\tilde h)$, the bound 
\begin{equation}
 \fP\bigpar{\Omega_1(\tilde h)^c}
 \leqs c_0 \bigpar{\q + \hbn_{\q}} \e^{-\kappa_0\tilde h^2/\sigma^2}
\end{equation} 
follows as a particular case of~\eqref{eq:bound_Il_2}, applied separately to 
all $\hbn$ of size $\abs{\hbn} = 1$. We now choose $\tilde h$ is such a way 
that 
\begin{equation}
 \kappa_0 \tilde h^2 
 = \kappa_m \frac{2^{(\q-q_0)/m} 2^{2\abs{\alpha}q_0/m}}
{(q_0+c_1)^2}h^2\;,
\end{equation} 
so that $\fP(\Omega_1(\tilde h)^c)$ and $\fP(\Omega_2(h, q_0)^c)$ are of 
comparable size. 
This allows us to bound the quantity 
\begin{equation}
 \sup_{t\in I_\ell}
 \norm{\delta_{q_0}\ph(t,\cdot)}_{L^2}
 \leqs 
 \sup_{t\in I_\ell}
 \norm{\delta_{q_0}\hat\ph(t,\cdot)}_{L^2} + 
 \sup_{t\in I_\ell}
 \norm{\delta_{q_0}\ph(t,\cdot) - \delta_{q_0}\hat\ph(t,\cdot)}_{L^2} 
\end{equation} 
on $\Omega_1(\tilde h)\cap\Omega_2(h, q_0)$. 
By Proposition~\ref{prop:phi_hatphi}, we have 
\begin{equation}
 \sup_{t\in I_\ell}
 \norm{\delta_{q_0}\ph(t,\cdot) - \delta_{q_0}\hat\ph(t,\cdot)}_{L^2}
 \leqs 2^{q_0} (c_1\gamma_0 m)^{[\hbn]} 
 \sum_{\hbl:\hbl \leqs\lfloor\hbn/2\rfloor}
 \frac{\hbn!}{\hbl!(\hbn-2\hbl)!}
 \biggpar{\frac{\hat c(t)}{2}}^{\abs{\hbl}} 
 \bigpar{2^{2\q}\tilde h}^{\abs{\hbn}-2\abs{\hbl}}\;.
\end{equation} 
Note the relation
\begin{align}
 \sum_{\hbl:\hbl \leqs\lfloor\hbn/2\rfloor}
 \frac{\hbn!}{\hbl!(\hbn-2\hbl)!} a^{\hbl} b^{\hbn - 2\hbl} 
 &= \prod_{q\geqs0}
 \sum_{\hbl_q \leqs \lfloor\hbn_q/2\rfloor}
 \frac{\hbn_q!}{\hbl_q!(\hbn_q-2\hbl_q)!} a^{\hbl_q} b^{\hbn_q - 2\hbl_q} \\
 &\leqs \prod_{q\geqs0} \frac{\hbn_q!}{(\hbn_q/2)!}
 \bigpar{\sqrt{a} + b}^{\hbn_q} 
 \leqs 2^{\abs{\hbn}}\bigpar{\sqrt{a} + b}^{\abs{\hbn}}\;.
\end{align} 
Since $\abs{\hbn} = m$, it follows that 
\begin{equation}
 \sup_{t\in I_\ell}
 \norm{\delta_{q_0}\ph(t,\cdot)}_{L^2}
 \leqs \frac12 \frac{2^{\abs{\alpha}q_0}h^m}{K_m(q_0) 2^{(\q - q_0)/2}}
 + 2^{q_0} (c_1\gamma_0 m)^m 
 2^m\bigpar{\hat c(t) + 2^{2\q}\tilde h}^m
\end{equation} 
holds on $\Omega_1(\tilde h)\cap\Omega_2(h, q_0)$. 
Choosing $\gamma_0$ such that both summands are equal 
yields the result. 
\end{proof}

\begin{cor}
\label{cor_1} 
We have 
\begin{multline}
 \biggprob{\sup_{t\in I}
 \bignorm{\delta_{q_0}(\prod_{q\geqs0}
\Wick{\delta_q\psi(t,\cdot)^{\hbn_q}})}_{L^2} > 
\frac{\hbn!}{m!}h_\hbn^m2^{\abs{\alpha} q_0}} \\
 \leqs \frac{T}{\eps}\tilde C_m(\hbn)q_0^{-m}
 2^{(m+3)\q}
 \exp\biggset{-\bar\kappa_m
\frac{2^{(\q-q_0)/m} 2^{2\abs{\alpha}q_0/m}}
{(q_0+c_1)^2}\frac{h^2}{\sigma^2} }\;,
\end{multline}
where $\tilde C_m(\hbn) = Q_m^{-1}\bar C_m(\hbn)$. 
\end{cor}
\begin{proof}
It suffices to sum the previous estimate over all $\ell\in\set{1,\dots,L}$, 
where $L$ has been introduced in~\eqref{eq:def_partition}. 
\end{proof}


\subsection{Summing over $q_0$ and $\hbn$}

Replacing the bound obtained in Corollary~\ref{cor_1} 
in~\eqref{eq:bound_Besovnorm_2}, we get 
\begin{equation}
P_m(h) \leqs \frac{T}{\eps}\sum_{q_0\geqs0}
q_0^{-m} \sum_{\abs{\hbn}=m} 
\tilde C_m(\hbn)2^{(m+3)\q}\exp\biggset{-\beta(m,q_0)2^{(\q-q_0)/m}}\;,
\end{equation}
where 
\begin{equation}
 \beta(m,q_0) 
= \bar\kappa_m\frac{ 
2^{2\abs{\alpha}q_0/m}}{(q_0+c_1)^2}\frac{h^2}{\sigma^2}\;.
\end{equation} 
We will first perform the sum over $\hbn$. To this end, we write 
\begin{equation}
\label{Km(q0)}
\bar K_{m,b}(q_0) = \sum_{\abs{\hbn}=m} 
\q^b 2^{(m+3)\q}\exp\biggset{-\beta(m,q_0)2^{(\q-q_0)/m}}\;,
\end{equation}
The following lemma is obtained in a similar way as Lemma~\ref{somme_ngras}. We 
give its proof in Appendix~\ref{appendix_B}. 

\begin{lemma}
\label{lem:sum_hbn} 
There are numerical constants $c_1, \beta_0$ such that for all $\beta(m,q_0) 
\geqs\beta_0$, one has the bound 
\begin{equation}
\bar K_{m,b}(q_0) \leqs 
c_1 q_0^{m+b}m^m 2^{(m+3)q_0} \e^{-\beta(m,q_0)}\;.
\end{equation}
\end{lemma}

Using the expression for $\bar C_m(\hbn)$ given in 
Proposition~\ref{prop:delta_q0_phi_main}, we thus obtain 
\begin{align}
P_m(h) 
&\leqs \frac{T}{\eps} Q_m^{-1} 
\sum_{q_0\geqs0} q_0^{-m} 
\Bigbrak{(2+\e^{m-1}+c_0m) \bar K_{m,0}(q_0) + \bar K_{m,1}(q_0)} \\
&\leqs \frac{T}{\eps} c_1 m^m Q_m^{-1} 
\sum_{q_0\geqs0} 
\Bigbrak{2+\e^{m-1}+c_0m+c_0q_0} 2^{(m+3)q_0} \e^{-\beta(m,q_0)}\;.
\label{eq:Pm(h)_final} 
\end{align}
It remains to perform the sums over $q_0$. These are of the form 
\begin{equation}
 \sum_{q_0\geqs0} f(q_0)\;, \qquad 
 f(x) = x^b 2^{ax} \exp\biggset{-\gamma\frac{2^{\abs{\alpha}x/m}}{(x+c_1)^2}}\;, 
 \qquad
 a = m+3\;, \qquad 
 \gamma = \bar\kappa_m \frac{h^2}{\sigma^2}\;.
\end{equation} 
One checks that $f$ is decreasing, so that one has the upper bound 
\begin{equation}
 \sum_{q_0\geqs0} f(q_0) \leqs f(0) + f(1) + \int_1^\infty f(x)\6x\;.
\end{equation} 
The terms $f(0)$ and $f(1)$ are both exponentially small in $h^2/\sigma^2$. To 
evaluate the integral, we can absorb the constant $c_1$ in $\gamma$, and the 
term $x^b$ into $2^{ax}$, by changing slightly the definitions of $\gamma$ and 
$a$. We first consider the case where the term $x^2$ in the denominator is 
absent, where the changes of variables $y = 2^{\abs{\alpha}x/m}$ and $z = 
\gamma y$ yield 
\begin{align}
\label{eq:integral_Gamma} 
 \int_1^\infty 2^{ax} \e^{-\gamma 2^{\abs{\alpha}x/m}} \6x
 &= \frac{m}{\abs{\alpha}\log2} \int_{2^{\abs{\alpha}/m}}^\infty 
 y^{am/\abs{\alpha}-1} \e^{-\gamma y} \6y \\
 &\leqs \frac{m}{\abs{\alpha}\log2}
 \frac{1}{\gamma^{\lambda+1}} \int_\gamma^\infty z^\lambda\e^{-z}\6z\;,
\end{align} 
with $\lambda = am/(2\abs{\alpha})$. The asymptotics of the incomplete Gamma 
function shows that 
\begin{equation}
 \int_1^\infty 2^{ax} \e^{-\gamma 2^{\abs{\alpha}x/m}} \6x
 \lesssim \frac{m}{\abs{\alpha}} \frac{\e^{-\gamma}}{\gamma} 
 \lesssim \frac{m}{\abs{\alpha}} \e^{-\gamma/2}\;.
\end{equation} 
In order to incorporate the effect of the denominator $x^2$, we use the upper 
bound 
\begin{equation}
 x^2 2^{-\abs{\alpha}x/m} \leqs \frac{4\e^{-2}}{(\log2)^2} 
\frac{m^2}{\abs{\alpha}^2}\;.
\end{equation} 
We can thus bound the integral of $f$ by the 
integral~\eqref{eq:integral_Gamma}, with $\gamma$ multiplied by a constant 
times $\abs{\alpha}^2/m^2$, and $\alpha$ divided by $2$. In other words, we get 
\begin{equation}
 \int_1^\infty f(x)\6x 
 \leqs c \frac{m}{\abs{\alpha}}\e^{-\kappa \abs{\alpha}^2\gamma/m^2}\;.
\end{equation} 
Replacing this in~\eqref{eq:Pm(h)_final} yields a bound on $P_m(h)$, completing 
the proof of Theorem~\ref{thm:stoch_convolution} in the case of a constant 
linearisation $a(t)$. 


\subsection{The case of a general linearisation $a(t)$}

Recall that we actually want to consider the more general linear equation 
\eqref{eq:stoch_convolution} 
\begin{equation}
\label{eq:stoch_convolution_t} 
 \6\tilde\psi(t,x) 
 = \frac{1}{\eps}\bigbrak{\Delta\tilde\psi(t,x) + a(t) \tilde\psi(t,x)} \6t + 
\frac{\sigma}{\sqrt{\eps}} \6W(t,x)\;, 
\end{equation} 
where here $a(t)$ satisfies~\eqref{eq:a(t)}, and we write $\psi = \tilde\psi$ to 
avoid confusion in the notations. Projecting~\eqref{eq:stoch_convolution_k_t} on 
the $k$th basis vector $e_k$, we obtain 
\begin{equation}
\label{eq:stoch_convolution_k_t} 
 \6\tilde\psi_k(t) 
 = \frac{1}{\eps}a_k(t)\tilde\psi_k(t) \6t + 
\frac{\sigma}{\sqrt{\eps}} \6W_k(t)\;,
\end{equation} 
where $a_k(t)=-\mu_k + a(t)$ and the $\set{W_k(t)}_{t\geqs0}$ are 
the same independent Wiener processes as before. The solution of 
\eqref{eq:stoch_convolution_k_t} with the same initial condition $\psi_k(0)$ 
as in~\eqref{stoch_convolution_k}, is given by 
\begin{equation}
\label{stoch_convolution_k2} 
\tilde \psi_k(t) =  \e^{\tilde\alpha_k(t)/\eps}\psi_k(0) + 
\frac{\sigma}{\sqrt{\eps}}\int_0^t \e^{\tilde\alpha_k(t,t_1)/\eps} \6W_k(t_1)\;,
\end{equation} 
where 
\begin{equation}
 \tilde \alpha_k(t,t_1) 
 = \alpha_k(t,t_1) + \int_{t_1}^t (1 + a(t_2))\6t_2
 = \alpha_k(t,t_1) + \Order{\abs{t - t_1}}\;.
\end{equation} 
For given $q_0$, we use the same partition of $[0,T]$ into intervals $I_l$ as 
before. On each interval, we can write 
\begin{equation}
\hat \psi_k(t) := \e^{\tilde\alpha_k(u_{l+1},t)/\eps} \tilde \psi_k(t)
 =\e^{\tilde\alpha_k(u_{l+1})/\eps}\psi_k(0) + \frac{\sigma}{\sqrt{\eps}}\int_0^t 
\e^{\tilde\alpha_k(u_{l+1},t_1)/\eps} \6W_k(t_1)\;.
\end{equation} 
$\hat \psi_k(t)$ is again a martingale, so that its supremum over 
the interval $I_l$ can be estimated as before. Note however that the variance 
of the associated sums over $k\in\cA_q$ is not exactly equal to $\hat c_q(t)$: 
it is rather of the form
\begin{equation}
 \hat V_q(t) = \hat c_q(t) \bigbrak{1 + \Order{\gamma_0 2^{-2(\q - q)}}}\;.
\end{equation} 
Therefore, the Wick powers of this martingale with respect to $\hat c_q(t)$ are 
not martingales. One can however estimate the supremum of the Wick powers with 
the correct variance as before, and then compare the two types of Wick powers. 
In fact, we want to bound the supremum over $t\in I_l$ of 
\begin{equation}
\bignorm{\delta_{q_0}(\prod_{q\geqs0}H_{\hbn_q}(\delta_q\tilde\psi(t,\cdot), c_q))}_{L^2}\;.
\end{equation} 
By the binomial formula introduced in Lemma~\ref{lem:hermite_binomial}, we obtain 
\begin{align}
\prod_{q\geqs0}H_{\hbn_q}&(\delta_q\tilde\psi(t,x), c_q) \\
={}& 
\prod_{q\geqs0}H_{\hbn_q}(\delta_q\hat\psi(t,x)+\delta_q\tilde\psi(t,
x)-\delta_q\hat\psi(t,x), \hat V_q(t)+c_q-\hat V_q(t))\;\\
={}&  \sum_{0\leqs\abs{\hbl}\leqs\abs{\hbn}} 
\binom{\hbn}{\hbl}\prod_{q\geqs0}H_{\hbl_q}(\delta_q\tilde\psi(t,
x)-\delta_q\hat\psi(t,x), c_q-\hat V_q(t))
H_{\hbn_q-\hbl_q}(\delta_q\hat\psi(t,x), \hat V_q(t))\;\\
={}&  \prod_{q\geqs0}H_{\hbn_q}(\delta_q\hat\psi(t,x), \hat V_q(t)) \;\\
&{}+ \sum_{1\leqs\abs{\hbl}\leqs\abs{\hbn}} 
\binom{\hbn}{\hbl}\prod_{q\geqs0}H_{\hbl_q}
(\delta_q\tilde\psi(t,x)-\delta_q\hat\psi(t,x), c_q-\hat V_q(t))
H_{\hbn_q-\hbl_q}(\delta_q\hat\psi(t,x), \hat V_q(t))\;.
\end{align} 
Therefore,  by the triangle inequality we get
\begin{align}
\bignorm{\delta_{q_0}(\prod_{q\geqs0}&H_{\hbn_q}(\delta_q\tilde\psi(t,\cdot), 
c_q))}_{L^2} \\
\leqs{}& 
\bignorm{\delta_{q_0}(\prod_{q\geqs0}H_{\hbn_q}(\delta_q\hat\psi(t,\cdot), \hat 
V_q(t)))}_{L^2}\;\\
&{}+ \sum_{1\leqs\abs{\hbl}\leqs\abs{\hbn}} 
\binom{\hbn}{\hbl}2^{q_0}\bignorm{\prod_{q\geqs0}H_{\hbl_q}
(\delta_q\tilde\psi(t,\cdot)-\delta_q\hat\psi(t,\cdot), c_q-\hat 
V_q(t))}_{L^2}\\
& \phantom{+} {}\times
\bignorm{\prod_{q\geqs0}H_{\hbn_q-\hbl_q}(\delta_q\hat\psi(t,\cdot), \hat 
V_q(t))}_{L^2}\;, 
\end{align} 
where the last inequality is a rough bound, obtained by Cauchy--Schwarz's 
inequality. The first norm can be bounded as in~\eqref{eq:bound_Il_2}, and one 
can see the last norm as a particular case of it. In the same spirit as in the 
proof of Proposition~\ref{prop:phi_hatphi}, for $\abs{\hbl} \geqs 1$ one can 
bound  
\begin{equation}
\bignorm{\prod_{q\geqs0}H_{\hbl_q}(\delta_q\tilde\psi(t)-\delta_q\hat\psi(t), 
c_q-\hat V_q(t))}_{L^2}\leqs \bigpar{c_1\gamma_0}^{\abs{\hbl}/2}\sum_{\hbp:\hbp 
\leqs\lfloor\hbl/2\rfloor}A_{\hbl\hbp}\hat c(t)^{\abs{\hbp}}
\prod_{q\geqs0}\bignorm{\delta_q\hat\psi(t,\cdot)}^{\hbl_q-2\hbp_q}_{L^2}\;,
\end{equation}
with
\begin{equation}
\hat c(t) = \sup_{q\geqs0} \hat c_q(t)
\qquad\text{and}\qquad 
A_{\hbl\hbp} = \frac{\hbl!}{2^{\abs{\hbp}} \hbp! (\hbl-2\hbp)!}
2^{2\q(\abs{\hbl} - 2\abs{\hbp})}\;.
\end{equation} 
Notice that 
\begin{equation}
\bigabs{ c_q-\hat V_q(t)}= \bigabs{\hat c_q(t)}
\bigabs{\e^{-2\alpha_{k(q)}(u_{l+1},t)/\eps} - 1 - \Order{\gamma_0 2^{-2(\q - 
q)}}} \leqs c_0\gamma_0. 
\end{equation}
Choosing $\gamma_0$ small enough in the sense taken in Proposition 
\ref{prop:delta_q0_phi_main} yields the result.

Finally, similar results holds for processes $\tilde\psi(t,x)$ with another 
initial condition. In particular, when it is equal to $0$, one can proceed 
exactly in the same way and bound the variance of the associated martingale by 
$\hat V_q(t)$.


\section{Proof of the concentration results}
\label{sec:proofs_concentration}

\subsection{Proof of Proposition~\ref{prop:stable_det}}

The proof of Proposition~\ref{prop:stable_det} is almost the same as the proof 
of Proposition 2.3 in~\cite{Berglund_Nader_22}, the only difference being that 
$x$ belongs to the two-dimensional torus. We give here only a few hints of the 
proof. 
Recall that the proposition concerns the deterministic equation
\begin{equation}
 \6\phi(t,x) = \frac{1}{\eps}\bigbrak{\Delta\phi(t,x) + F(t, \phi(t,x))} \6t\;, 
\end{equation} 
where $F$ satisfies \eqref{eq:def_F} and Assumption \ref{assump:stable}. 
We consider the difference $\psi(t)=\phi(t)-\phi^*(t)e_0$. Using 
Taylor's formula to expand $ F(t, \psi(t)+\phi^*(t)e_0)$, we obtain that $\psi$ 
satisfies the equation 
\begin{equation}
\label{eq:SPDE_det_psi}
\eps \partial_t \psi(t,x) = \Delta\psi(t,x) + a(t)\psi(t,x) + b(t,\psi(t,x))-\eps\frac{\6}{\6t}\phi^*(t)e_0(x)\;,
\end{equation} 
where
\begin{align}
a(t) &= \partial_\phi F(t,\phi^*(t)e_0)\;, \\
b(t,\psi) &= \frac{1}{2}\partial_\phi^2 F\big(t,\phi^*(t)+\theta 
\psi\big)\psi^2 \qquad\text{ for some } \theta \in \brak{0,1}\;.
\end{align}
We define the Lyapounov function 
\begin{equation}
\label{eq:Lyapunov}
 V(\psi) 
 = \frac12 \norm{\psi}_{H^1}^2 = \frac12 \norm{\psi}_{L^2}^2 + 
\frac{L^2}{2\pi^2} \norm{\nabla\psi}_{L^2}^2\;. 
\end{equation} 
Its time derivative satisfies
\begin{equation}
\label{eq:dtV_det_stable1} 
 \eps\frac{\6}{\6t} V (\psi) \leqs 2a(t) V(\psi) + \pscal{\psi}{b(t,\psi)} 
 - \frac{L^2}{\pi^2}\pscal{\Delta\psi}{b(t,\psi)} - \eps \frac{\6}{\6t} 
\phi^*(t) \pscal{\psi}{e_0}\;.
\end{equation} 
We introduce for a fixed  $C_0>0$, $\bar\tau$, the first exist-time from 
the set $\bigset{ V(\psi(t,\cdot)) \leqs C_0}$.
Then, we bound the different terms in \eqref{eq:dtV_det_stable1} similarly to 
the proof in~\cite{Berglund_Nader_22}. We arrive at the relation
\begin{equation}
 \eps \dot V \leqs -\frac12 C_1V + \eps C_2 V^{1/2}
\end{equation} 
for all $t \leqs \bar\tau$, and some constants $C_1, C_2> 0$. Using Gronwall's 
inequality, we find that there exists a particular solution  satisfying 
$V(\psi(t,\cdot)) = \Order{\eps^2}$ for all $t \leqs \bar\tau$. The result 
extends for all $t\in I$.

\subsection{Proof of Lemma~\ref{lem:F0}}

The binomial formula for Hermite polynomials yields 
\begin{equation}
 \Wick{\phi(t,x)^i} = \sum_{j=0}^i \binom{i}{j} \bar\phi(t,x)^{i-j} 
\Wick{\phi_0(t,x)^j}\;. 
\end{equation} 
Using the definition~\eqref{eq:def_F} of $F$ and swapping the sums, we obtain 
\begin{equation}
 \Wick{F_0(t,x,\phi_0(t,x))}
 = \sum_{j=1}^n \biggbrak{\sum_{i=j}^n \binom{i}{j} A_i(t) \bar\phi(t,x)^{i-j}}
 \Wick{\phi_0(t,x)^j}
\end{equation} 
(note that the terms $j=0$ cancel). This proves the claim for the terms with 
$j\geqs 2$. For $j=1$, we note that 
\begin{equation}
 a(t)  = \partial_\phi F(t,\phi^*(t))
 = \sum_{i=1}^n i A_i(t) \phi^*(t)^{i-1}\;.
\end{equation} 
Rearranging terms yields the claimed result. Proposition~\ref{prop:stable_det} 
shows that $\bar\phi(t,\cdot) \in H^1$, and that $\smallnorm{\hat 
A_1(t,\cdot)}_{H^1} = \Order{\eps}$. By 
\cite[Th\'eor\`eme~7]{Bourdaud_calcul_symbolique}, powers of $\bar\phi$ belong 
to $H^1$ as well. 
\qed

\subsection{Proof of Theorem~\ref{thm:phi1}}

Recall that $\phi_1(t,x)$ solves the equation  
\begin{equation}
\label{eq:phi1_repeated} 
 \6\phi_1(t,x) 
 = \frac{1}{\eps} \bigbrak{\Delta\phi_1(t,x) + a(t)\phi_1(t,x) 
 + \Wick{b(t,x,\psi(t,x)+\phi_1(t,x))}}\6t\;,
\end{equation} 
where $\psi(t,x)$ is the stochastic convolution, and 
\begin{equation}
\label{eq:Wickb} 
 \Wick{b(t,x,\psi(t,x)+\phi_1(t,x))}
 = \sum_{j=1}^n \hat A_j(t,x) 
 \sum_{\ell=0}^j \binom{j}{\ell} \phi_1(t,x)^{j-\ell} \Wick{\psi(t,x)^\ell}\;.
\end{equation} 
Assume that $\psi(t,\cdot)\in\Besovspace{\alpha}{2}{\infty}$ and 
$\phi_1(t,\cdot)\in\Besovspace{\beta}{2}{\infty}$ for all $t\in[0,T]$. 
The bound~\eqref{eq:bound2_product_Besov} on products in Besov spaces shows 
that 
\begin{equation}
\smallnormBesov{\phi_1(t,\cdot)^{j-\ell}
\Wick{\psi(t,\cdot)^\ell}}{(2(j-\ell)+1)\alpha}{2}{\infty}
\leqs M_1 \normBesov{\phi_1(t,\cdot)}{\beta}{2}{\infty}^{j-\ell} 
\smallnormBesov{\Wick{\psi(t,\cdot)^\ell}}{\alpha}{2}{\infty}
\end{equation} 
for a constant $M_1 = M_1(\beta,\alpha,n)$, provided $\beta\geqs1+2\alpha$. 

We treat separately the term $j=1$ in the sum~\eqref{eq:Wickb} and the 
remaining terms. For $j=1$, we use the fact that $\hat A_1(t,\cdot) \in 
\Besovspace{1}{2}{\infty}$ has a norm of order $\eps$ 
and~\eqref{eq:bound1_product_Besov} to obtain that  
\begin{equation}
 \smallnormBesov{\hat A_1(t,\cdot)(\phi_1(t,\cdot) + 
\psi(t,\cdot))}{\bar\alpha}{2}{\infty} 
 \leqs M_2 \eps \Bigpar{\normBesov{\phi_1(t,\cdot)}{\beta}{2}{\infty} 
 + \normBesov{\psi(t,\cdot)}{\alpha}{2}{\infty}}
\end{equation} 
for any $\bar\alpha < \alpha$. For $j\geqs2$, we have a similar bound, but 
without the factor $\eps$, since the $\hat A_j$ are of order $1$ in $H^1$. 
Now let $h, H \in(0,1]$ be constants such that 
\begin{equation}
 \max_{1\leqs\ell\leqs n} 
 \smallnormBesov{\Wick{\psi(t,\cdot)^\ell}}{\alpha}{2}{\infty}
 \leqs h\;, \qquad 
 \normBesov{\phi_1(t,\cdot)}{\beta}{2}{\infty} \leqs H\;.
\end{equation} 
Summing over $j$, since $h, H\leqs 1$ we get the existence of constants $M_3, 
M_4$ such that  
\begin{align}
 \smallnormBesov{\Wick{b(t,x,\psi+\phi_1)}}{\bar\alpha}{2}{\infty}
 &\leqs M_2\eps(H+h) + M_3 \biggbrak{H^2 + \sum_{j=2}^n\sum_{\ell=1}^j 
H^{j-\ell}h^\ell} \\
&\leqs M_4(H+h)(H+h+\eps)
\end{align} 
for any $\bar\alpha < (2n-1)\alpha$. 
We now fix a $\gamma < \bar\alpha + 2$ and introduce the stopping time 
\begin{equation}
 \tau = \inf\Bigset{t\in[0,T] \colon 
\normBesov{\phi_1(t,\cdot)}{\gamma}{2}{\infty} > H}\;.
\end{equation} 
Then we have
\begin{align}
 \prob{\tau < T}
 \leqs{}& 
 \biggprob{\exists\ell\in\set{1,\dots,n} \colon \sup_{t\leqs T}
 \smallnormBesov{\Wick{\psi(t,\cdot)^\ell}}{\alpha}{2}{\infty} > h^\ell} \\
 &{}+ \biggprob{\tau < T,\; \sup_{t\leqs T}
 \smallnormBesov{\Wick{\psi(t,\cdot)^\ell}}{\alpha}{2}{\infty} \leqs h^\ell\;
 \forall \ell\in\set{1,\dots,n}}\;.
 \label{eq:proof_phi1_decomp} 
\end{align}
The first term on the right-hand side can be bounded using 
Theorem~\ref{thm:stoch_convolution}. As for the second term, we use the fact 
that under the condition on the Wick powers of the stochastic convolution being 
small, the Schauder estimate given in Corollary~\ref{prop:Schauder_cor} yields 
\begin{equation}
 \normBesov{\phi_1(T\wedge\tau, \cdot)}{\gamma}{2}{\infty} 
 \leqs M\eps^{-\nu}(H+h)(H+h+\eps)\;, 
 \qquad \nu = 1 - \frac{\gamma-\bar\alpha}{2}
\end{equation} 
for a constant $M$. Choosing first $H = 2M\eps^{-\nu}h(h+\eps)$, and 
then $\eps$ small enough and $h < h_0\eps^\nu$ for a sufficiently small $h_0$, 
one can ensure that $(H+h)(H+h+\eps) < 2h(h+\eps)$, so that the second 
probability is actually equal to zero. 

To conclude the proof, we first pick a $\gamma < 2$, and then $\bar\alpha \in  
(\gamma-2, 0)$, and finally $\alpha \in (\frac{\bar\alpha}{2n-1}, 0)$. We also 
require that $\beta \leqs \gamma$, which is possible by choosing $\beta = 
1+2\alpha$ and asking that $\alpha \leqs -\frac12(1-\gamma)$. This yields the 
claimed result, thanks to the embedding 
$\Besovspace{\gamma}{2}{\infty} \hookrightarrow 
\Besovspace{\gamma-1}{\infty}{\infty} = \cC^{\gamma-1}$.   
\qed

\subsection{Proof of Theorem~\ref{thm:phi1perp}}

The proof is very similar to the proof of Theorem~\ref{thm:phi1}, so we only 
comment on the differences. Given $\alpha < 0$ and $H_\perp > 0$, we introduce 
stopping times 
\begin{align}
 \tau_\psi(h) &= \inf\Bigset{t\in[0,T] \colon \max_{1\leqs\ell\leqs 3} 
 \smallnormBesov{\Wick{\psi_\perp(t,\cdot)^\ell}}{\alpha}{2}{\infty} > h}\;, \\
 \tau_\perp(H_\perp) &= \inf\Bigset{t\in[0,T] \colon 
 \smallnormBesov{\phi_1^\perp(t,\cdot)}{\gamma}{2}{\infty} > 
H_\perp}\;.
\end{align}
For any $\bar\alpha < 5\alpha$, writing $\tau = \tau_\psi(h) \wedge 
\tau_\perp(H_\perp) \wedge \tau_0(H_0)$, one obtains the existence of a constant 
$M$ such that, for any $t \leqs \tau$, one has 
\begin{equation}
 \smallnormBesov{\Wick{F_\perp(\psi_\perp(t,\cdot), 
\phi_1^0(t,\cdot), \phi_1^\perp(t,\cdot))}}{\bar\alpha}{2}{\infty} 
 \leqs M(h + H_0 + H_\perp)^3\;.
\end{equation} 
Using Duhamel's formula to write the solution of~\eqref{eq:bif_phi1perp} in 
integral form, and the Schauder estimate in Corollary~\ref{prop:Schauder_cor} 
(adapted to the eigenvalues of the new linear part), one obtains 
\begin{equation}
 \smallnormBesov{\phi_1^\perp(t \wedge \tau)}{\gamma}{2}{\infty}
 \leqs M_1 \eps^{-\nu} (h + H_0 + H_\perp)^3
\end{equation} 
for $\nu < 1 - \frac12(\gamma - \bar\alpha)$ and a constant $M_1(\nu) > 0$, 
provided $1 + 2\alpha\geqs\gamma$. 
Then it suffices to decompose the probability as 
in~\eqref{eq:proof_phi1_decomp}. Choosing $H_\perp = 2M_1\eps^{-\nu}(h + 
H_0)^3$ and $h + H_0$ of order $\eps^{\nu/2}$ yields the result. 
\qed

\subsection{Proof of Theorem~\ref{thm:phi10}}

The proof is essentially the same as the proof 
of~\cite[Theorem~2.10]{BG_pitchfork} and~\cite[Proposition~4.7]{BG_pitchfork}, 
except that one has to account for the effect of the extra term 
$F_0(\psi_\perp,\phi_1^0,\phi_1^\perp)$ in the equation. The solution 
of~\eqref{eq:bif_phi10} admits the integral representation 
\begin{equation}
\label{eq:proof_thm_phi10} 
 \phi_1^0(t) 
 = \phi^\circ(t) 
 + \frac{1}{\eps} \int_0^t \e^{\alpha(t,t_1)/\eps}
 \bigbrak{-(\phi_1^0(t_1))^3 + 
F_0(\psi_\perp(t_1),\phi_1^0(t_1),\phi_1^\perp(t_1))} \6t_1\;,
\end{equation} 
where $\phi^\circ$ is the solution of the linear equation~\eqref{eq:phicirc}. 
Proposition~4.3 in~\cite{BG_pitchfork} provides a similar estimate 
as~\eqref{eq:pitchfork_bound1} for $\phi^\circ$. Furthermore, up to time 
$\tau_{\cB_-(h_-)}\wedge\tau_\psi(h) \wedge \tau_\perp(H_\perp)$, we have the 
bound 
\begin{equation}
 \bigabs{-(\phi_1^0(t_1))^3 
+ F_0(\psi_\perp(t_1),\phi_1^0(t_1),\phi_1^\perp(t_1))}
 \leqs M
 \biggpar{\frac{h_-}{\eps^{1/4}} + h + H_\perp}^3
\end{equation} 
for a constant $M$. The integral over $t_1$ in~\eqref{eq:proof_thm_phi10} 
yields an extra factor $1/\sqrt{\eps}$. This allows to bound the supremum of 
$\phi_1^0(t)$ in terms of the supremum of $\phi^\circ(t)$ on the event 
\begin{equation}
 \Omega_{h,H_\perp} = \bigset{\tau_\psi(h) \wedge \tau_\perp(H_\perp) > 
\tau_{\cB_-(h_-)}}\;.
\end{equation} 
The probability of the complement $\smash{\Omega_{h,H_\perp}^c}$ can be 
estimated by Theorems~\ref{thm:stoch_convolution} and~\ref{thm:phi1perp}. 
Choosing $H_0 = h_-\eps^{-1/4}$, $h = h_-$ and $H_\perp = \eps^{-\nu}(h + 
H_0)^3$, one finds that $\fP(\Omega_{h,H_\perp}^c)$ is negligible with respect 
to the probability of $\phi^\circ$ leaving $\cB_-(h_-[1 - 
\Order{h_-^2/\eps}])$, which proves~\eqref{eq:pitchfork_bound1}.  

To prove~\eqref{eq:pitchfork_bound2}, we use the fact that for 
$t \leqs \tau_{\cB_+(h_+)}\wedge\tau_\psi(h) \wedge \tau_\perp(H_\perp)$, one 
has 
\begin{align}
\bigabs{-(\phi_1^0(t_1))^3 
+ F_0(\psi_\perp(t_1),\phi_1^0(t_1),\phi_1^\perp(t_1))}
& \leqs M_1
 \biggpar{\frac{h_+}{a(t_1)^{1/2}} + h + H_\perp}^3 \\
& \leqs M_2
 \biggpar{\frac{h_+}{\eps^{1/4}} + h + H_\perp}^3
\end{align} 
for constants $M_1, M_2$. This motivates the choice 
$h = H_\perp = \eps^{-1/4}h_+$. Proceeding as in the proof 
of~\cite[Proposition~4.7]{BG_pitchfork}, one obtains 
\begin{equation}
 \bigprob{\tau_{\cB_+(h_+)} > t, \tau_\psi(h) \vee \tau_\perp(H_\perp) \geqs t}
 \leqs \frac{h_+}{\sigma} \exp\biggset{-\kappa \frac{\alpha(t,t^*)}{\eps}}\;.
\end{equation} 
The second term on the right-hand side of~\eqref{eq:pitchfork_bound2} bounds 
$\prob{\tau_\psi(h) \vee \tau_\perp(H_\perp) < t}$.
\qed


\appendix
\section{Besov spaces}
\label{app:Besov} 

\begin{lemma}
\label{lem:Besov_test} 
Let $\eta: \T^2\to\R$ be a compactly supported function of class $\cC^1$, with 
$\norm{\eta}_{\cC^1} = 1$. For any $p\in[2,\infty]$ and any $\rho\in(0,1]$, let 
\begin{equation}
 \eta^{(p)}_{\rho}(x) = \frac{1}{\rho^{2(1-1/p)}} 
\eta\biggpar{\frac{x}{\rho}}\;.
\end{equation} 
Then $\smallnorm{\eta^{(p)}_{\rho}}_{L^r} = \norm{\eta}_{L^r}$ for all 
$\rho\in(0,1]$, where $r = (1-1/p)^{-1}$ is the H\"older conjugate of $p$.
Moreover, for any $\psi \in \Besovspace{\alpha}{p}{\infty}$ with 
$\alpha\in(-1,0)$, and any $q\in\N_0$, one has 
\begin{equation}
 \bigabs{\pscal{\psi}{\eta^{(p)}_{2^{-q_0}}}}
 \lesssim 2^{\abs{\alpha}q_0} \normBesov{\psi}{\alpha}{p}{\infty}\;.
\end{equation} 
\end{lemma}
\begin{proof}
We have 
\begin{equation}
 \smallnorm{\eta^{(p)}_{\rho}}_{L^r}^r 
 = \frac{1}{\rho^2} \int_{\T^2} \eta\biggpar{\frac{x}{\rho}}^r \6x 
 = \int_{\T^2} \eta(y)^r \6y
 = \norm{\eta}_{L^r}^r\;,
\end{equation} 
where we have used the change of variables $x = \rho y$, and the fact that the 
integration domain does not change because $\eta$ is compactly supported. For 
the same reason, we have 
\begin{equation}
 \bigabs{\pscal{e_k}{\eta^{(p)}_\rho}} 
 = \rho^{2/p} \biggabs{\int_{\T^2} \e^{-\icx\rho k\cdot y} \eta(y)\6y}
 \leqs \frac{\rho^{2/p}}{1\vee\rho^2\abs{k_1k_2}}\;,
\end{equation} 
where we have used one integration by parts if $\rho^2\abs{k_1k_2} \leqs 1$. 
In particular, for $k\in\cA_q$ and $\rho = 2^{-q_0}$, this yields 
\begin{equation}
 \bigabs{\pscal{e_k}{\eta^{(p)}_{2^{-q_0}}}} 
 \leqs \frac{2^{-2q_0/p}}{1\vee2^{2(q-q_0)}}\;.
\end{equation} 
Using H\"older's inequality, we obtain 
\begin{align}
 \bigabs{\pscal{\delta_q\psi}{\eta^{(p)}_{2^{-q_0}}}}
 &= \bigabs{\pscal{\delta_q\psi}{\delta_q\eta^{(p)}_{2^{-q_0}}}} 
 \leqs \norm{\delta_q\psi}_{L^p} 
\smallnorm{\delta_q\eta^{(p)}_{2^{-q_0}}}_{L^r}\\
 &\leqs \norm{\delta_q\psi}_{L^p} 
 \biggpar{\sum_{k\in\cA_q} \bigabs{\pscal{e_k}{\eta^{(p)}_{2^{-q_0}}}}^p}^{1/p} 
\\
 &\leqs 2^{\abs{\alpha}q} \normBesov{\psi}{\alpha}{p}{\infty}
 \frac{2^{2(q-q_0)/p}}{1\vee2^{2(q-q_0)}}\;.
\end{align}
The result then follows by summing over all $q\in\N_0$, noticing that this sum 
is dominated by the term $q = q_0$. 
\end{proof}

\begin{prop}[Schauder estimate on the heat kernel]
\label{prop:Schauder} 
Let $g \in \Besovspace{\alpha}{2}{\infty}$ for some $\alpha\in\R$, and let 
$\e^{t\Delta}$ denote the heat kernel. Then there exists a constant 
$M_0$ depending on $\beta-\alpha$ such that 
\begin{equation}
 \normBesov{\e^{t\Delta}g}{\beta}{2}{\infty}
 \leqs M_0 t^{-\frac{\beta-\alpha}{2}} 
 \normBesov{g}{\alpha}{2}{\infty}
\end{equation} 
holds for all $t\geqs0$ and all $\beta \leqs \alpha+2$. 
\end{prop}
\begin{proof}
Denoting by $\mu_k$ the eigenvalues of the Laplacian 
(cf.~\eqref{eq:ev_Laplacian}) and by $(g_k)_{k\in\Z^2}$ the Fourier 
coefficients of $g$, there is a constant $c>0$ such that
\begin{equation}
 \norm{\delta_q(\e^{t\Delta}g)}_{L^2}^2 
 = \sum_{k\in\cA_q} \e^{-2\mu_kt} \abs{g_k}^2 
 \leqs \e^{-c2^{2q}t} \norm{\delta_q g}_{L^2}^2 
 \leqs \e^{-c2^{2q}t} 2^{-2q\alpha} \normBesov{g}{\alpha}{2}{\infty}^2
\end{equation} 
for all $q$. Therefore, 
\begin{equation}
 \normBesov{\e^{t\Delta}g}{\beta}{2}{\infty}
 \leqs\sup_{q\geqs0} 2^{q(\beta-\alpha)} \e^{-\frac12 c2^{2q}t} 
 \normBesov{g}{\alpha}{2}{\infty}\;.
\end{equation} 
Now we observe that for any $\gamma\geqs0$,  
\begin{equation}
 2^{q(\beta-\alpha)} \e^{-\frac12 c2^{2q}t}
 = 2^{q(\beta-\alpha-2\gamma)}t^{-\gamma}
 (2^{2q}t)^\gamma \e^{-\frac12 c2^{2q}t}
 \leqs M_0(2\gamma) 2^{q(\beta-\alpha-2\gamma)}t^{-\gamma}
\end{equation} 
by boundedness of the map $x\mapsto x^\gamma\e^{-x}$.
Choosing $\gamma = \frac{\beta-\alpha}{2}$ yields the result. 
\end{proof}

\begin{cor}[Schauder estimate on convolutions with the heat kernel]
\label{prop:Schauder_cor} 
Let $g(t) \in \Besovspace{\alpha}{2}{\infty}$ for all $t\in [0,T]$, where 
$\alpha\in\R$. Let $\phi$ be the solution of 
\begin{equation}
\label{eq:phi_Schauder} 
 \6\phi(t,x) = \frac{1}{\eps} \bigbrak{\Delta\phi(t,x) + a(t)\phi(t,x) + 
g(t,x)}\6t\;,
\end{equation}
starting from $0$, where $a\in\cC^1([0,T],\R_-)$ is bounded away from $0$ 
(cf.~\eqref{eq:a(t)}). Then $\phi(t) \in \Besovspace{\beta}{2}{\infty}$ for 
all $\beta < \alpha+2$ and all $t\in [0,T]$, and there is a constant 
$M = M(\beta-\alpha)$ such that 
\begin{equation}
 \normBesov{\phi(t,\cdot)}{\beta}{2}{\infty} 
 \leqs M \eps^{\frac{\beta-\alpha}{2}-1}
 \sup_{t_1\in[0,T]} \normBesov{g(t_1,\cdot)}{\alpha}{2}{\infty} 
\end{equation} 
holds for all $\beta < \alpha+2$ and all $t\in[0,T]$. 
\end{cor}
\begin{proof}
The solution of~\eqref{eq:phi_Schauder} can be written as 
\begin{equation}
 \phi(t,x) = \frac{1}{\eps} \int_0^t \e^{\alpha(t,t_1)/\eps} 
 \bigpar{\e^{\frac{t-t_1}{\eps}\Delta}g}(t_1,x)\6t_1\;, 
\end{equation} 
where $\alpha(t,t_1) = \int_{t_1}^t a(t_2)\6t_2$ is negative whenever $t\geqs 
t_1$. Therefore 
\begin{align}
 \normBesov{\phi(t,\cdot)}{\beta}{2}{\infty} 
 &\leqs 
 \frac{1}{\eps} \int_0^t 
\smallnormBesov{\bigpar{\e^{\frac{t-t_1}{\eps}\Delta}g}
(t_1,\cdot)}{\beta}{2}{\infty} \6t_1 \\
 &\leqs 
 \frac{1}{\eps} M_0(\beta-\alpha)\int_0^t
 \biggpar{\frac{t-t_1}{\eps}}^{-\frac{\beta-\alpha}{2}}
 \normBesov{g(t_1,\cdot)}{\alpha}{2}{\infty} \6t_1 \\
 &\leqs
 \eps^{\frac{\beta-\alpha}{2}-1} M_0(\beta-\alpha) 
 \sup_{t_1\in[0,T]} \normBesov{g(t_1,\cdot)}{\alpha}{2}{\infty} 
 \int_0^t (t-t_1)^{-\frac{\beta-\alpha}{2}}\6t_1\;.
\end{align} 
The integral is bounded whenever $\beta < \alpha + 2$. 
\end{proof}

\section{Wick calculus}
\label{app:Wick}

This appendix summarises some properties of Hermite polynomials and Wick 
calculus needed in this work. Proofs of these properties can be found, for 
instance, in the monographs~\cite{nualart2006malliavin,Peccati_Taqqu_book}, the 
lecture notes~\cite{Hairer_Malliavin}, and Section~4.2.2 and Appendix~D 
of~\cite{B-SPDE_book}. 

Hermite polynomials with variance $C$ have been introduced in 
Section~\ref{ssec:Wick}. Some of the above references consider the special case 
$C = 1$, but results for that case can easily be converted into results for the 
general case by using the scaling property
\begin{equation}
 \label{eq:Hermite_scaling}
 H_n(x; C) = C^{n/2} H_n(C^{-1/2}x; 1)\;.
\end{equation} 
The first $n$ Hermite polynomials and the monomials $1, \dots, x^n$ both form a 
basis of the vector space of polynomials of degree $n$, where the change of 
basis is given by the formulas 
\begin{align}
\label{eq:Hermite_change_basis} 
H_n(x; C) &= \sum_{\ell = 0}^{\lfloor n/2 \rfloor} a_{n\ell} C^\ell 
x^{n-2\ell}\;, 
& a_{n\ell} &= \frac{(-1)^\ell n!}{2^\ell \ell! (n-2\ell)!}\;, \\
x^n &= \sum_{\ell = 0}^{\lfloor n/2 \rfloor} b_{n\ell} C^\ell 
H_{n-2\ell}(x; C)\;, 
& b_{n\ell} &= \frac{n!}{2^\ell \ell! (n-2\ell)!} = \abs{a_{n\ell}}\;.
\end{align}
The Hermite polynomials admit the generating function 
\begin{equation}
\label{eq:hermite_generating} 
G(t,x;C) :=
\e^{tx - C t^2/2} = \sum_{n=0}^\infty \frac{t^n}{n!} H_n(x;C)\;,
\end{equation} 
which can be used to establish the following identity.

\begin{lemma}[Expectation of products of Wick powers]
\label{lem:hermite_moments} 
Let $X$ and $Y$ be jointly Gaussian centred random variables, of respective 
variance $C_1$ and $C_2$. Then for any $n,m\geqs0$, one has 
\begin{equation}
\label{eq:Wick_identity} 
 \bigexpec{H_n(X;C_1) H_m(Y;C_2)} = 
 \begin{cases}
  n! \bigexpec{XY}^n & \text{if $n=m$\;,} \\
  0 & \text{otherwise\;.}
 \end{cases}
\end{equation} 
\end{lemma}

Another consequence of the expression~\eqref{eq:hermite_generating} of the 
generating function is the following binomial formula.

\begin{lemma}[Binomial formula for Hermite polynomials]
\label{lem:hermite_binomial}
For any $x,y\in\R$, $C_1,C_2\geqs0$ and $n\in\N_0$, 
\begin{equation}
\label{eq:Hermite_binomial} 
 H_n(x+y; C_1+C_2) = \sum_{m=0}^n \binom{n}{m} 
 H_m(x; C_1)H_{n-m}(y; C_2)\;.
\end{equation} 
\end{lemma}

A direct consequence is Lemma~\ref{lem:Hermite_martingale}, which we recall 
here. We recall that $\hat\psi_k(t)$ is the martingale introduced 
in~\eqref{eq:psi_k_hat}, and that $\hat v_k(t)$ is its variance defined 
in~\eqref{eq:v_k_hat}.

\begin{lemma}[Martingale property]
For any $m\geqs1$, $H_m(\hat\psi_k(t); v_k(t))$ is a martingale with respect to 
the canonical filtration $\cF_t$ of the Wiener process $(W_k(t))_{t\geqs0}$.
\end{lemma} 
\begin{proof}
We write $H_m(\hat\psi_k(t); v_k(t)) = H_m(\hat\psi_k(t))$ in order not to 
overload the notation. For any $0\leqs s < t$, we have 
\begin{equation}
\label{expec_martingale}
\E\bigbrak{H_m(\hat\psi_k(t))\bigm|\cF_s}=\E\Bigbrak{
H_m\bigpar{\hat\psi_k(s)+(\hat\psi_k(t)-\hat\psi_k(s))}\bigm|\cF_s}.
\end{equation}
By the binomial formula~\eqref{eq:Hermite_binomial} and additivity of the 
$v_k$, 
\begin{equation}
H_m\bigpar{\hat\psi_k(s)+(\hat\psi_k(t)-\hat\psi_k(s))}=\sum_{n=0}^m 
\binom{m}{n} H_n(\hat\psi_k(s))H_{m-n}(\hat\psi_k(t)-\hat\psi_k(s))\;.
\end{equation}
We replace this expression in \eqref{expec_martingale}.  Since 
$H_n(\hat\psi_k(s))$ is $\cF_s$-measurable and 
$H_{m-n}(\hat\psi_k(t)-\hat\psi_k(s))$ is independent of $\cF_s$ we obtain 
\begin{align}
\E\bigbrak{H_m(\hat\psi_k(t))\bigm|\cF_s}
={}&\sum_{n=0}^m 
\binom{m}{n}\E\bigbrak{H_n(\hat\psi_k(s))H_{m-n}
(\hat\psi_k(t)-\hat\psi_k(s))\bigm|\cF_s}\;\\
={}&\sum_{n=0}^m 
\binom{m}{n}H_n(\hat\psi_k(s))\E\bigbrak{H_{m-n}(\hat\psi_k(t)-\hat\psi_k(s))}
\;\\
={}&H_m(\hat\psi_k(s))\;.
\end{align}
The last equality is due to the fact that $m$th Hermite polynomials are 
centred variables for $m\geqs1$ and for $m=n$, 
$H_0(\hat\psi_k(t)-\hat\psi_k(s))=1$. 
\end{proof} 

The following generalisation of the binomial 
formula~\eqref{eq:Hermite_binomial} is obtained by induction. 

\begin{lemma}[Multinomial formula for Hermite polynomials]
\label{lem:hermite_multinomial}
Let $(a_q)_{q\geqs0}$ be a sequence of real numbers in $\ell^2$. Then for any 
convergent sequence $(x_q)_{q\geqs0}$, one has 
\begin{equation}
 H_m\biggpar{\sum_{q\geqs0}x_q; \sum_{q\geqs0} a_q^2} 
 = \sum_{\abs{\hbn} = m} 
 \frac{m!}{\hbn!} \prod_{q\geqs0} H_{\hbn_q}(x_q; a_q^2)\;,
\end{equation} 
where the sum runs over all $\hbn\in\N_0^{\N_0}$ such that $\abs{\hbn} := 
\sum_{q\geqs0} \hbn_q = m$, and $\hbn! := \prod_{q\geqs0} \hbn_q!$.
\end{lemma}

Given a set $\set{\psi_q}_q$ of independent centred Gaussian random variables, 
one defines the $m$th homogeneous Wiener chaos $\cH_m$ as the vector space 
generated by all Wick powers of the $\psi_q$ of total degree $m$, that is, all 
\begin{equation}
 \Phi_{\hbn} = \prod_{q\geqs0} H_{\hbn_q}\bigpar{\psi_q; \variance(\psi_q)}
\end{equation} 
with $\abs{\hbn} = m$. Then one has the following result on equivalence of 
norms, which is a consequence of hypercontractivity of the Ornstein--Uhlenbeck 
semigroup. See for instance~\cite[Theorem~4.1]{Hypercontractivity} 
or~\cite[Theorem~1.4.1]{nualart2006malliavin}. 

\begin{lemma}[Equivalence of moments]
\label{lemma:Nelson}
Let $X$ be a random variable, belonging to the $m$-th homogeneous Wiener
chaos. Then for any $p \geqs 1$ one has
\begin{equation}
\label{eq:A-Nelson} 
 \bigexpec{X^{2p}} \leqs 
(2p-1)^{mp} \bigexpec{X^2}^p\;.
\end{equation}
\end{lemma}

\section{Some technical proofs for Section~\ref{sec:proof_stoch_convolution}}
\label{appendix_B} 


\subsection{Proof of Lemma~\ref{lemma:Xq4}}

We divide the proof of the lemma into the following two parts. 

\begin {lemma}
\label{Xq4}
For any $q_0\geqs 0$, $t\in I_l$ and $\hbn$, one has 
\begin{equation}
\label{bound_Xq2_proof}
\E\bigbrak{X_{\hbn}^2} 
= \hbn!\sum_{\substack{\kq_1,\kq_2,...,\kq_{\hbn_q}\in\cA_{q}\; 
\forall q \\ 
\sum_{q\geqs0}\sum_{i=1}^{\hbn_q}\kq_i\in\cA_{q_0}}}\prod_{q\geqs0}
\prod_{i=1}^{\hbn_q}\hat v_{\kq_i}(t)\;.
\end{equation}
\end{lemma}
\begin{proof}
Let $\varphi(t,\cdot) = \displaystyle
\prod_{q\geqs0}\Wick{\delta_q\hat\psi(t,\cdot)^{\hbn_q}}$. The $L^2$-norm of its 
projection on $\cA_{q_0}$ is given by
\begin{equation}
\bignorm{\delta_{q_0}\varphi(t,\cdot)}_{L^2}^2
=\sum_{k\in\cA_{q_0}}\abs{(P_k\varphi)(t,\cdot)}^2\;,
\end{equation}
where $(P_k\varphi)(t,x)$ is the projection of $\ph$ on the $k$th Fourier 
basis vector $e_k(x)$, given by 
\begin{equation}
(P_k\varphi)(t,x)=\int_{\T^2}e_{-k}(x_1)\varphi(t,x_1)\6x_1 e_k(x)\;.
\end{equation}
Therefore,
\begin{equation}
\E\Bigbrak{\bignorm{\delta_{q_0}\varphi(t,\cdot)}_{L^2}^2}=\sum_{k\in\cA_{q_0}}\E\Bigbrak{\abs{(P_k\varphi)(t,\cdot)}^2}.
\end{equation}
For a fixed $k\in\cA_{q_0}$, we have 
\begin{align}
\E\Bigbrak{\abs{(P_k\varphi)(t,x)}^2}={}&\E\Bigbrak{\int_{\T^2}\int_{\T^2}e_{-k}(x_1)\varphi(t,x_1)e_k(x_2)\bar\varphi(t,x_2)\6x_1\6x_2 e_k(x)e_{-k}(x)}\;\\
={}&\int_{\T^2}\int_{\T^2}e_{-k}(x_1-x_2)
\E\bigbrak{\varphi(t,x_1)\bar\varphi(t,x_2)}\6x_1\6x_2\;, 
\end{align}
where 
\begin{align}
\E\bigbrak{\varphi(t,x_1)\bar\varphi(t,x_2)}
&= \E\Bigbrak{\prod_{q\geqs0}\Wick{\delta_q\hat\psi(t,x_1)^{\hbn_q}}
\overline{\Wick{\delta_q\hat\psi(t,x_2)^{\hbn_q}}}\,} \\
&= \prod_{q\geqs0}\E\bigbrak{\Wick{
\delta_q\hat\psi(t,x_1)^{\hbn_q}}
\overline{\Wick{\delta_q\hat\psi(t,x_2)^{\hbn_q}}}\,}\;,
\end{align}
since the projections $\delta_{q}$ and $\delta_{q'}$ are independent for 
$ q\ne q'$.
By Lemma~\ref{lem:hermite_moments}, we get
\begin{equation}
\E\bigbrak{\Wick{\delta_q\hat\psi(t,x_1)^{\hbn_q}}
\overline{\Wick{\delta_q\hat\psi(t,x_2)^{\hbn_q}}}\,}
= \hbn_q!\E\bigbrak{\delta_q\hat\psi(t,x_1)
\overline{\delta_q\hat\psi(t,x_2)}\,}^{\hbn_q}\;,
\end{equation}
where
\begin{align}
\E\bigbrak{\delta_q\hat\psi(t,x_1)
\overline{\delta_q\hat\psi(t,x_2)}\,}
={}&\sum_{k_1, k_2\in\cA_q} 
\E\bigbrak{\hat\psi_{k_1}(t)\overline{\hat\psi_{k_2}(t)}}
e_{k_1}(x_1)e_{-k_2}(x_2)\;\\
={}&\sum_{k_1,k_2\in\cA_q}\hat v_{k_1}(t)
\delta_{k_1,k_2}e_{k_1}(x_1)e_{-k_2}(x_2)\;\\
={}&\sum_{k_1\in\cA_q}\hat v_{k_1}(t)e_{k_1}(x_1-x_2)\;.
\end{align}
Therefore, 
\begin{align}
\E\bigbrak{\varphi(t,x_1)\bar\varphi(t,x_2)}
={}&\prod_{q\geqs0}\hbn_q!\Bigpar{
\sum_{k_1\in\cA_q}\hat v_{k_1}(t)e_{k_1}(x_1-x_2)}^{\hbn_q}\;\\
={}&\biggpar{\prod_{q\geqs0}\hbn_q!} \prod_{q\geqs0}
\biggpar{\sum_{k_1,...,k_{\hbn_q}\in\cA_q}\hat v_{k_1}(t)\cdots 
\hat v_{k_{\hbn_q}}(t) e_{k_1+...k_{\hbn_q}}(x_1-x_2)}\;.
\end{align}
Integrating over $x_1$ and $x_2$, we get
\begin{align}
\E\Bigbrak{\abs{(P_k\varphi)(t,x)}^2}
={}&\hbn!\sum_{\kq_1,...,\kq_{\hbn_q}\in\cA_q,\ \forall 
q}\prod_{q\geqs0}\prod_{i=1}^{\hbn_q}\hat v_{\kq_i}(t)\;\\
{}&\times\int_{\T^2}\int_{\T^2}e_{-k}(x_1-x_2)e_{\sum_{q\geqs0}\kq_1+...+\kq_{\hbn_q}}(x_1-x_2)\6x_1\6x_2\;\\
={}&\hbn!\sum_{\kq_1,...,\kq_{\hbn_q}\in\cA_q,\ \forall 
q}\prod_{q\geqs0}\prod_{i=1}^{\hbn_q}\hat v_{\kq_i}(t)
\mathbbm{1}_{\bigset{\sum_{q\geqs0}\sum_{i=1}^{\hbn_q}\kq_i=k}}\;\\
={}&\hbn!\sum_{\substack{\kq_1,...,\kq_{\hbn_q}\in\cA_{q},\ \forall q \\ 
\sum_{q\geqs0}\sum_{i=1}^{\hbn_q}\kq_i=k}}\prod_{q\geqs0}\prod_{i=1}^{\hbn_q}
\hat v_{\kq_i}(t)\;.
\end{align}
Summing over $k_0\in\cA_{q_0}$ yields the claimed result. 
\end{proof}

\begin {lemma}
\label{Xq2}
There exists a numerical constant $C_0$ such that
\begin{equation}
\E\bigbrak{X_{\hbn}^2} 
\leqs C_0^m \hbn!\sigma^{2m}\frac{2^{2q_0}}{2^{2\q}}
\leqs C_0^m m!\sigma^{2m}\frac{2^{2q_0}}{2^{2\q}}\;.
\end{equation}
\end{lemma}

\begin{proof}
We have to evaluate the sum given by~\eqref{bound_Xq2_proof}. 
Recall that $q_1 < q_2 < \dots < \q$ denote the indices of the $[\hbn]$ nonzero 
entries of $\hbn$, and that there is a numerical constant $c_0$ such that
\begin{equation}
 \hat v_k(t) \leqs \frac{c_0 \sigma^2}{1 + \norm{k}^2}
\end{equation}
for all $t\in I_l$. 
For a fixed $k_0\in\cA_{q_0}$, we  get the bound
\begin{equation}
S_{\hbn,k_0} := 
\sum_{\substack{\kq_1,...,\kq_{\hbn_q}\in\cA_{q},\ \forall q \\ 
\sum_{q\geqs0}\sum_{i=1}^{\hbn_q}\kq_i=k_0}}\prod_{q\geqs0}\prod_{i=1}^{\hbn_q}
\hat v_{\kq_i}(t)
\leqs \sum_{\substack{\kq_1,...,\kq_{\hbn_q}\in\cA_{q},\ \forall q \\ 
\sum_{q\geqs0}\sum_{i=1}^{\hbn_q}\kq_i=k_0}}\prod_{q\geqs0}\prod_{i=1}^{\hbn_q}
\frac{c_0\sigma^2}{1+\bignorm{\kq_i}^2}\;.
\end{equation}
Note that 
\begin{equation}
\prod_{q\geqs0}\prod_{i=1}^{\hbn_q}(c_0\sigma^2)
= \prod_{q\geqs0}(c_0\sigma^2)^{\hbn_q}
= (c_0\sigma^2)^{\sum_{q\geqs0}\hbn_q}
= (c_0\sigma^2)^{m}\;,
\end{equation}
and that we can write 
\begin{equation}
\label{kdeq}
k^{(q_1)}_1
=k_0-\sum_{j=2}^{\brak{\hbn}}k^{(q_j)}_1-\sum_{j=1}^{\brak{\hbn}}\sum_{i=2}^{
\hbn_q}k^{(q_j)}_i\;.
\end{equation}
Since $\bignorm{k^{(q_1)}_i}<\bignorm{k^{(q_2)}_i}<...<\bignorm{k^{(\q)}_i}$ and 
$\bignorm{k_0}\leqs\bignorm{k^{(\q)}_i}$, by the second triangle inequality, we 
get 
\begin{align}
\label{ineq_trian}
\biggnorm{k_0-\sum_{j=2}^{\brak{\hbn}}k^{(q_j)}_1-\sum_{j=1}^{\brak{\hbn}}
\sum_{i=2}^{\hbn_q}k^{(q_j)}_i}
\geqs{}& 
\biggabs{\bignorm{k_0}-\sum_{j=2}^{\brak{\hbn}}
\bignorm{k^{(q_j)}_1}-\sum_{j=1}^
{\brak{\hbn}}\sum_{i=2}^{\hbn_q}\bignorm{k^{(q_j)}_i}}\;\\
\geqs{}&\sum_{j=2}^{\brak{\hbn}}\bignorm{k^{(q_j)}_1} + 
\sum_{j=1}^{\brak{\hbn}}\sum_{i=2}^{\hbn_q}\bignorm{k^{(q_j)}_i}
-\bignorm{k_0}\;\\
\geqs{}& c\norm{k^{(\q)}_1}
\end{align}
for a numerical constant $c>0$. 
Replacing $k^{(q_1)}_1$ by \eqref{kdeq} and bounding its norm by 
\eqref{ineq_trian}, we obtain 
\begin{align}
S_{\hbn,k_0}
\lesssim{}&(c_0\sigma^2)^{m}\prod_{j=1}^{\brak{\hbn}}
\Biggpar{\sum_{k^{(q_j)}_1,\dots,k^{(q_j)}_{\hbn_{q_j}}\in\cA_{q_j}}
\frac{1}{\bignorm{k_i^{(q_j)}}^2}}^{\hbn_{q_j}}
\sum_{k^{(\q)}_1\in\cA_{\q}}\frac{1}{\bignorm{k^{(\q)}_1}^4}\;.
\end{align}
For a fixed $q_j$,  we view these sums as Riemann sums, and integrating using 
polar coordinates yields 
\begin{equation}
\sum_{k^{(q_j)}_1,\dots,k^{(q_j)}_{\hbn_{q_j}}\in\cA_{q_j}}
\frac{1}{\bignorm{k^{(q_j)}_i}^2} \lesssim 
\int_{2^{q_j-1}}^{2^{q_j}}\frac{r}{r^2}\6r \leqs \log(2^{q_j})-\log(2^{q_j-1})= 
\log(2)\;,
\end{equation}
and 
\begin{equation}
\sum_{k^{(\q)}_1\in\cA_{\q}}\frac{1}{\bignorm{k^{(\q)}_1}^4} \lesssim 
\int_{2^{\q-1}}^{2^{\q}}\frac{r}{r^4}\6r \lesssim\frac{1}{2^{2\q}}\;.
\end{equation}
We conclude that 
\begin{equation}
S_{\hbn,k_0}
\lesssim
(c_0\sigma^2)^m\prod_{j=1}^{\brak{\hbn}}
(c_1\log(2))^{\hbn_{q_j}}\frac{1}{2^{2\q}}
= C_0^m\sigma^{2m}\frac{1}{2^{2\q}}
\end{equation}
for some numerical constant $c_1, C_0$. The result follows again by summing 
over $k\in\cA_{q_0}$.
\end{proof}

\subsection{Proof of Lemma~\ref{somme_ngras}}

We decompose the sum~\eqref{eq:Kmq0} as 
\begin{equation}
\label{eq:Kq0_sum} 
K_m(q_0) =\sum_{\brak{\hbn}=1}^mS_m(\brak{\hbn},0)\;,
\end{equation}
where for $a\in\set{1,\dots, m}$ and $b\in\N_0$, we define 
\begin{equation}
\label{eq:def_Sm} 
 S_m(a,b)
 =\sum_{\substack{\hbn:\abs{\hbn}=m\\
 \brak{\hbn}=a, \; \q + \hbn_{\q} \geqs q_0}}
\frac{m!}{\hbn!}
\frac{\q^b}{2^{(\q-q_0)/2}}\;.
\end{equation}
We will estimate this sum by induction on $a$, for arbitrary $b\in\N_0$. For 
$a=1$, the only possible $\hbn$ are those with one component, say $q$, equal to 
$m$, and all other components equal 
to $0$. Therefore, 
\begin{equation}
 S_m(1,b) \leqs \sum_{q\geqs (q_0 - m)\vee 0} \frac{q^b}{2^{(q-q_0)/2}}
\end{equation} 
(since one must have $q\geqs0$). The sum can be computed via the 
inequality 
\begin{equation}
 \sum_{q=0}^\infty (q+1)^b z^q \leqs \frac{b!}{(1-z)^{b+1}}\;,
\end{equation} 
valid for any $z\in[0,1)$ and $b\in\N$, which follows directly from 
the definitions of the polylogarithm function and Eulerian numbers. Setting 
$1/\sqrt{2}=z$, we have 
\begin{align}
 S_m(1,b) 
 &\leqs \sum_{q\geqs (q_0 - m)\vee 0} q^b z^{q-q_0} 
 = z^{-m} \sum_{p\geqs0\vee(m-q_0)} (q_0-m+p)^b z^{p} \\
 &\leqs z^{-m} \sum_{\ell = 0}^b \binom{b}{\ell} q_0^{b-\ell} 
 \sum_{p\geqs 0} p^\ell z^p \\
 &\leqs z^{-m} \frac{b!}{1-z} \sum_{\ell = 0}^b \binom{b}{\ell} 
 \frac{q_0^{b-\ell}}{(1-z)^\ell} \\
 &= 2^{m/2} b!\, c_1(q_0 + c_1)^b\;,
\end{align}
where $c_1 = (1 - z)^{-1} = 2+\sqrt{2}$. 

Assume now that $a\geqs1$ and $[\hbn] = a + 1$. We decompose $\hbn = \hbn_1 + 
\hbn_2$, with $[\hbn_1] = a$ and $[\hbn_2] = 1$, and $\abs{\hbn_1} = m_1$,  
$\abs{\hbn_2} = m_2$ with $m_1 + m_2 = m$. We may assume that the largest 
nonzero component of $\hbn$ appears in $\hbn_1$, so that $\qq{1} = \q$ and 
$\qq{2} = q < \q$. It follows that 
\begin{align}
 S_m(a+1, b) 
 &= \sum_{m_1 + m_2 = m} \sum_{q < \q}
 \sum_{\substack{\abs{\hbn_1} = m_1 \\ [\hbn_1] = a \\ 
 \qq{1} + \hbn_{\qq{1}} \geqs q_0}}
 \frac{m!}{\hbn_1!m_2!} \frac{\qq{1}^b}{2^{(\qq{1} - q_0)/2}} \\
 &\leqs \sum_{m_1=1}^{m-1} \binom{m}{m_1} S_{m_1}(a, b+1)\;,
\end{align} 
where we have bounded the sum over $q$ by $\q = \qq{1}$. 
It is then straightforward to show by induction that 
\begin{equation}
 S_m(a, b) \leqs c_1 (\sqrt{2} + a - 1)^m (q_0 + c_1)^{a+b-1} (a+b-1)!
\end{equation} 
for all $a, b$. In particular, 
\begin{equation}
 S_m(a, 0) \leqs c_1 (\sqrt{2} + a - 1)^m (q_0 + c_1)^{a-1} (a-1)!\;.
\end{equation} 
Replacing this in~\eqref{eq:Kq0_sum} yields the result, with $c_2 = \sqrt{2} - 
1$. 
\qed

\subsection{Proof of Proposition~\ref{prop:phi_hatphi}}

Using the relation~\eqref{eq:Hermite_change_basis} between Wick polynomials 
and monomials, we get 
\begin{align}
\prod_{q\geqs0}H_{\hbn_q}(\delta_q\psi(t,\cdot); c_q) =&{} 
\prod_{q\geqs0}\biggpar{\sum_{\hbl_q = 0}^{\lfloor 
\hbn_q/2 \rfloor} 
a_{\hbn_q\hbl_q} c_q^{\hbl_q} 
\bigpar{\delta_q\psi(t,\cdot)}^{\hbn_q-2\hbl_q}}\;\\
=&{} \sum_{\hbl:\hbl\leqs\lfloor\hbn/2\rfloor}a_{\hbn\hbl} 
\prod_{\substack{q\geqs0\\\hbn_q>0}} c_q^{\hbl_q}
\bigpar{\delta_q\psi(t,\cdot)}^{\hbn_q-2\hbl_q}\;,
\end{align}
where
\begin{equation}
a_{\hbn\hbl} 
= \prod_{\substack{q\geqs0\\\hbn_q>0}}a_{\hbn_q\hbl_q} = 
\prod_{\substack{q\geqs0\\\hbn_q>0}}\frac{(-1)^{\hbl_q} 
\hbn_q!}{2^{\hbl_q} \hbl_q! (\hbn_q-2\hbl_q)!} = \frac{(-1)^{\abs{\hbl}} 
\hbn!}{2^{\abs{\hbl}} \hbl! (\hbn-2\hbl)!}\;.
\end{equation}
Recall that  
\begin{equation}
\delta_q\psi(t,x)=\sum_{k\in\cA_q}\psi_k(t)e_k(x)\;,
\end{equation}
which implies 
\begin{align}
\prod_{\substack{q\geqs0\\\hbn_q>0}} \bigpar{\delta_q\psi(t,x)}^{\hbn_q-2\hbl_q} 
=&{} 
\prod_{\substack{q\geqs0\\\hbn_q>0}}  
\biggpar{\sum_{k_1,...,k_{\hbn_q-2\hbl_q}\in\cA_q}\psi_{k_1}(t)...\psi_{k_{
\hbn_q-2\hbl_q}}(t)e_{k_1+...+k_{\hbn_q-2\hbl_q}}(x)}\;\\
=&{} \sum_{\kq_1,...,\kq_{\hbn_q-2\hbl_q}\in\cA_q\ \forall 
q}\prod_{\substack{q\geqs0\\\hbn_q>0}}\prod_{i=1}^{\hbn_q-2\hbl_q}   
\psi_{\kq_i}(t)e_{\sum_{q\geqs0\colon\hbn_q>0}\sum_{i=1}^{\hbn_q-2\hbl_q} 
\kq_i}(x)\;,
\end{align}
whose projection on the $k_0$th Fourier basis vector is given by 
\begin{align}
P_{k_0}\biggpar{\prod_{\substack{q\geqs0\\\hbn_q>0}} 
\bigpar{\delta_q\psi(t,x)}^{\hbn_q-2\hbl_q}}
=&{} 
\sum_{\cB(k_0)}\prod_{\substack{q\geqs0\\\hbn_q>0}}\prod_{i=1}^{\hbn_q-2\hbl_q} 
\psi_{\kq_i}(t)e_{k_0}(x)\;,
\end{align}
where the sum runs over all tuples 
$(\kq_1,...,\kq_{\hbn_q-2\hbl_q})_{q>0}$ 
in the set
\begin{equation}
 \cB(k_0) = 
 \biggset{\kq_1,...,\kq_{\hbn_q-2\hbl_q}\in\cA_q\ \forall 
q \colon \sum_{q\geqs0}\sum_{i=1}^{\hbn_q-2\hbl_q} 
\kq_i=k_0}\;.
\end{equation}
Similar relations hold with $\hat\psi(t,x)$. 
We now note that
\begin{equation}
\bignorm{\delta_{q_0}\bigpar{\varphi(t,\cdot) - \hat\varphi(t,\cdot)}}^2_{L^2} 
= \sum_{k_0\in\cA_{q_0}}\bigabs{\pscal{e_{k_0}}{{P_{k_0}\varphi(t,\cdot) - 
P_{k_0}\hat\varphi(t,\cdot)}}}^2\;,
\end{equation}
where 
\begin{multline}
 \pscal{e_{k_0}}{{P_{k_0}\varphi(t,\cdot) - 
P_{k_0}\hat\varphi(t,\cdot)}} \\
= \sum_{\hbl:\hbl\leqs\lfloor\hbn/2\rfloor}a_{\hbn\hbl} 
\sum_{\cB(k_0)} 
\biggbrak{\prod_{\substack{q\geqs0\\\hbn_q>0}}c_q^{\hbl_q}\prod_{i=1}^{
\hbn_q-2\hbl_q}   
\psi_{\kq_i}(t) - \prod_{\substack{q\geqs0\\\hbn_q>0}}\hat 
c_q(t)^{\hbl_q}\prod_{i=1}^{\hbn_q-2\hbl_q} 
\hat\psi_{\kq_i}(t)}\;.
\label{eq:proof_Pk0phiphihat1} 
\end{multline} 
Observe that 
\begin{align}
 \prod_{i=1}^{\hbn_q-2\hbl_q}   
\psi_{\kq_i}(t)
&= \exp\biggset{-\frac{1}{\eps}\sum_{i=1}^{ \hbn_q-2\hbl_q
}\alpha_{\kq_i}(u_{l+1},t)} \prod_{i=1}^{\hbn_q-2\hbl_q} 
\hat\psi_{\kq_i}(t)\;, \\
c_q^{\hbl_q} 
&\leqs \exp\biggset{-\frac{2}{\eps}\hbl_q\alpha_{k(q)}(u_{l+1},t)} \hat 
c_q(t)^{\hbl_q}
\end{align} 
for some $k(q) \in \cA_q$. The definition of the partition implies that 
$\abs{u_{l+1}-t}$ has order $2^{-2\q} \gamma_0\eps$, and therefore there is a 
numerical constant $c_0$ such that 
\begin{equation}
 -\frac{1}{\eps} \alpha_k(u_{l+1},t) \leqs c_0 2^{-2(\q - q)} \gamma_0
\end{equation} 
holds for all $k\in\cA_q$. Therefore, 
\begin{equation}
 -\frac{1}{\eps}\sum_{i=1}^{ \hbn_q-2\hbl_q
}\alpha_{\kq_i}(u_{l+1},t) -\frac{2}{\eps}\hbl_q\alpha_{k(q)}(u_{l+1},t)
\leqs c_0\gamma_0 \abs{\hbn_q} 2^{-2(\q - q)}\;. 
\end{equation} 
Replacing this in~\eqref{eq:proof_Pk0phiphihat1} yields 
\begin{multline}
 \bigabs{\pscal{e_{k_0}}{{P_{k_0}\varphi(t,\cdot) - 
P_{k_0}\hat\varphi(t,\cdot)}}} \\
\leqs \sum_{\hbl:\hbl\leqs\lfloor\hbn/2\rfloor}\bigabs{a_{\hbn\hbl}} 
\sum_{\cB(k_0)} 
\prod_{\substack{q\geqs0\\\hbn_q>0}}
\Bigpar{\e^{c_0\gamma_0 \abs{\hbn_q}2^{-2(\q - q)}} - 1}
{\prod_{\substack{q\geqs0\\\hbn_q>0}}
\biggpar{\hat c_q(t)^{\hbl_q}\prod_{i=1}^{\hbn_q-2\hbl_q} 
\bigabs{\hat\psi_{\kq_i}(t)}}}\;.
\label{eq:proof_Pk0phiphihat2} 
\end{multline} 
Since the exponent $c_0\gamma_0 \abs{\hbn_q}2^{-2(\q - q)}$ is bounded, we 
can write, for a numerical constant $c_1$,  
\begin{align}
\prod_{\substack{q\geqs0\\\hbn_q>0}}
\Bigpar{\e^{c_0\gamma_0 \abs{\hbn_q}2^{-2(\q - q)}} - 1}
&\leqs \prod_{\substack{q\geqs0\\\hbn_q>0}}
\Bigpar{c_1\gamma_0 \abs{\hbn_q}2^{-2(\q - q)}} \\
&\leqs \bigpar{c_1\gamma_0}^{[\hbn]}
\prod_{\substack{q\geqs0\\\hbn_q>0}} \hbn_q\;,
\end{align}
since the product of powers of $2$ is bounded by $1$ (in fact, it can even be 
bounded by $2^{-2([\hbn] - 1)}$, but this just decreases the constant $c_1$).
Since 
\begin{equation}
\bignorm{\delta_q\hat\psi(t,\cdot)}^2_{L^2} = 
\sum_{k\in\cA_q}\abs{\hat\psi_k(t)}^2,
\end{equation}
we have the rough bound 
\begin{equation}
 \abs{\hat\psi_k(t)}^2\leqs\bignorm{\delta_q\hat\psi(t,\cdot)}^2_{L^2} 
\quad\forall 
k\in\cA_q\;.
\end{equation}
Plugging the last bounds into~\eqref{eq:proof_Pk0phiphihat2}, we get  
\begin{multline}
 \bigabs{\pscal{e_{k_0}}{{P_{k_0}\varphi(t,\cdot) - 
P_{k_0}\hat\varphi(t,\cdot)}}} \\
\leqs \bigpar{c_1\gamma_0}^{[\hbn]}
\Biggpar{\prod_{\substack{q\geqs0\\\hbn_q>0}} \hbn_q}
\sum_{\hbl:\hbl\leqs\lfloor\hbn/2\rfloor}\bigabs{a_{\hbn\hbl}} 
\hat c(t)^{\abs{\hbl}}
\prod_{\substack{q\geqs0\\\hbn_q>0}}
\biggpar{\bignorm{\delta_q\hat\psi(t,\cdot)}^{\hbn_q-2\hbl_q }_{L^2}}
\# \cB(k_0)\;.
\label{eq:proof_Pk0phiphihat3} 
\end{multline} 
Finally, by counting the number of choices of the $\kq_i$, we obtain 
\begin{equation}
 \# \cB(k_0) \leqs 2^{2\q(\abs{\hbn} - 2\abs{\hbl})}\;.
\end{equation} 
This yields the claimed result, noticing that this bound is independent of 
$k_0$, so summing over all $k_0\in\cA_{q_0}$ only yields an extra factor 
$2^{2q_0}$ in the $L^2$-norm squared. 
\qed

\subsection{Proof of Lemma~\ref{lem:sum_hbn}}

We decompose the sum \label{Km(q0)} as 
\begin{equation}
\label{eq:barKq0_sum} 
\bar K_{m,b}(q_0) =\sum_{a=1}^mS_{m,m}(a, b)\;,
\end{equation}
where for $a\in\set{1,\dots, m}$ and $b\in\N_0$, we define
\begin{align}
\label{eq:def_Sm} 
 S_{m,m_0}(a,b)
 =&{}\sum_{\substack{\hbn:\abs{\hbn}=m\\ \brak{\hbn}=a, \; \q + \hbn_{\q} \geqs 
q_0}}\q^b2^{(m_0+3)\q}\exp\biggset{-\beta(m_0,q_0)2^{(\q-q_0)/m_0}}\;.
\end{align}
We will proceed similarly to the proof of Lemma~\ref{lem:sum_hbn}, and estimate 
this sum by induction on $a$, for arbitrary $b\in\N_0$. For $a=1$, the only 
possible $\hbn$ are those with one component, say $q$, equal to $m$, and all 
other components equal to $0$. Then $\q = q$, and writing $x_+ = x\vee0$ we get
\begin{align}
 S_{m,m_0}(1,b) \leqs&{} \sum_{q\geqs (q_0 - m)_+}
 q^b2^{(m_0+3)q}\exp\biggset{-\beta(m_0,q_0)2^{(q-q_0)/m_0}}\;\\
 ={}& \sum_{p\geqs(m-q_0)_+}(q_0-m+p)^b2^{(m_0+3)(p+q_0-m)}
\exp\biggset{-\beta(m_0,q_0)2^{(p-m)/m_0}}\;\\
={}&2^{(m_0+3)(q_0-m)}\sum_{p\geqs (m-q_0)_+}
(q_0-m+p)^b2^{(m_0+3)p}\exp\biggset{-\beta(m_0,q_0)2^{(p-m)/m_0} }\;.
\end{align} 
One checks that for $\beta(m_0,q_0)$ larger than a numerical constant of order 
$1$, 
the general term of this sum is decreasing in $p$. Estimating the sum by an 
integral, we get
\begin{equation}
S_{m,m_0}(1,b) \leqs 
c_1(q_0-m)^b2^{(m_0+3)(q_0-m)_+}\exp\biggset{-2^{-m/m_0}\beta(m_0,q_0)}
\end{equation}
for a numerical constant $c_1$. 
Assume now that $a\geqs1$ and $[\hbn] = a + 1$. We decompose $\hbn = \hbn_1 + 
\hbn_2$, with $[\hbn_1] = a$ and $[\hbn_2] = 1$, and $\abs{\hbn_1} = m_1$,  
$\abs{\hbn_2} = m_2$ with $m_1 + m_2 = m$. We may assume that the largest 
nonzero component of $\hbn$ appears in $\hbn_1$, so that $\qq{1} = \q$ and 
$\qq{2} = q < \q$. It follows that 
\begin{align}
 S_{m,m_0}(a+1,b) 
 &= \sum_{m_1 + m_2 = m} \sum_{q < \q}
 \sum_{\substack{\abs{\hbn_1} = m_1 \\ [\hbn_1] = a \\ 
 \qq{1} + \hbn_{\qq{1}} \geqs q_0}} \!\!
\q^b2^{(m_0+3)\q}\exp\biggset{-\beta(m_0,q_0)2^{(\q-q_0)/m_0}}\\
&\leqs \sum_{m_1 = 1}^{m-1} S_{m_1,m_0}(a,b+1) \;.
\end{align} 
where we have bounded the sum over $q$ by $\q = \qq{1}$. It is then 
straightforward to show by induction that 
\begin{equation}
 S_{m,m_0}(a, b) \leqs 
c_1 m^{a-1} q_0^{a+b-1} 2^{(m_0+3)q_0}\exp\biggset{-\beta(m_0,q_0)}
\end{equation} 
for all $a, b$. Summing over $a$ and setting $m=m_0$ yields the 
result.
\qed


\bibliographystyle{plain}
{\small \bibliography{SR_ref}}

{\small \tableofcontents}

\vfill

\bigskip\bigskip\noindent
{\small
Institut Denis Poisson (IDP) \\ 
Universit\'e d'Orl\'eans, Universit\'e de Tours, CNRS -- UMR 7013 \\
B\^atiment de Math\'ematiques, B.P. 6759\\
45067~Orl\'eans Cedex 2, France \\
{\it E-mail addresses: }
{\tt nils.berglund@univ-orleans.fr}, 
{\tt rita.nader@univ-orleans.fr}

\end{document}